\documentclass[letterpaper,12pt]{article}

\usepackage{microtype,booktabs}
\usepackage{amsmath,amssymb,amsfonts,amsthm,bm}
\usepackage{authblk}
\usepackage{mathrsfs}
\usepackage{indentfirst}
\usepackage[pdftex]{color,graphicx}
\usepackage[pdftex,bookmarks,unicode,colorlinks,linkcolor=blue]{hyperref}
\usepackage{tikz}
\usepackage{subfigure}
\usetikzlibrary{arrows, decorations.pathmorphing, backgrounds, positioning, fit, petri, automata,cd}
\usepackage{rotating}
\numberwithin{equation}{section}

\usepackage{todonotes}

\usepackage{tabularx} 
\usepackage{ragged2e} 
\usepackage{booktabs}

\title{\Large The stochastic Hamilton--Jacobi--Bellman equation on Jacobi structures}

\date{}
\author[$\dag$]{\normalsize Pingyuan Wei}
\author[$\ddag$]{Qiao Huang}
\author[$\S$]{Jinqiao Duan}
\affil[$\dag$]{\small School of Mathematics, Southeast University, Nanjing 211189, China; Beijing International Center for Mathematical Research, Peking University, Beijing 100871, China. (Email:pwei@seu.edu.cn)}
\affil[$\ddag$]{School of Mathematics, Southeast University, Nanjing 211189, China. (Email:qiao.huang@seu.edu.cn)}
\affil[$\S$]{Department of Mathematics \& Guangdong Provincial Key Laboratory of Mathematical and Neural Dynamical Systems, 
Great Bay University, Dongguan 523000, China.
(Email:duan@gbu.edu.cn)}

\newtheorem{thm}{Theorem}[section]
\newtheorem{theorem}{Theorem}[section]

\newtheorem{corollary}[thm]{Corollary}

\newtheorem{lemma}[thm]{Lemma}

\newtheorem{proposition}[thm]{Proposition}
\theoremstyle{remark}

\newtheorem{remark}[thm]{Remark}
\theoremstyle{definition}
\newtheorem{definition}[thm]{Definition}
\newtheorem{example}[thm]{Example}

\newcommand{\pt}{\partial}
\newcommand{\R}{\mathbb R}
\renewcommand{\d}{\boldsymbol{d}}

\definecolor{colorP}{HTML}{6600FF}
\definecolor{colorB}{HTML}{0099FF}
\definecolor{colorG}{HTML}{00CC00}
\definecolor{colorN}{HTML}{CC9900}
\definecolor{colorR}{HTML}{FF0000}

\begin{document}
\maketitle

\begin{abstract}

Jacobi structures are known to generalize Poisson structures, encompassing symplectic, cosymplectic, and Lie--Poisson manifolds. Notably, other intriguing geometric structures---such as contact and locally conformal symplectic manifolds---also admit Jacobi structures but do not belong to the Poisson category. In this paper, we employ global stochastic analysis techniques, initially developed by Meyer and Schwartz, to rigorously introduce stochastic Hamiltonian systems on Jacobi manifolds. We then propose a stochastic Hamilton--Jacobi--Bellman (HJB) framework as an alternative perspective on the underlying dynamics. We emphasize that many of our results extend the work of Bismut \cite{Bismut1980,Bismut1981} and L\'azaro-Cam\'i \& Ortega \cite{Lazaro2008,LazaroCam2009}. Furthermore, aspects of our geometric Hamilton--Jacobi theory in the stochastic setting draw inspiration from the deterministic contributions of Abraham \& Marsden \cite{Marsden1978}, de Le\'on \& Sard\'on \cite{deLeon2017}, Esen et al. \cite{Esen2021}, and related literature.

\end{abstract}

\tableofcontents

\section{Introduction}\label{sec1}
Stochastic geometric mechanics provides a comprehensive framework for the stochastic deformation of classical geometric mechanics, encompassing the study of symmetries, Hamiltonian and Lagrangian mechanics, and the Hamilton--Jacobi theory, among other areas. Over the past few decades, interest in this field has grown significantly (see \cite{Bismut1980,Bismut1981,Lazaro2008,LazaroCam2009,Cruzeiro2018,WeiChaoDuan2019,WeiWang2021,HuangZambrini2,HuangZambrini1} and references therein). This framework addresses the stochastic dynamics of diverse systems, including particles, rigid bodies, continuous media (e.g., fluids, plasmas, and elastic materials), and field theories (e.g., electromagnetism and gravity). The associated theories play essential roles in classical and quantum mechanics, control theory, and a wide range of applications across physics, engineering, chemistry, and biology \cite{Marsden1978,Arnold1989,Marsden1999,Applebaum2009,Gliklikh2011,Duan2015}.


In particular, interesting quick historical overviews as well as the motivations of the studies for stochastic Hamiltonian systems with continuous semimartingsles on Poisson manifolds can be found in L\'azaro-Cam\'i $\&$ Ortega \cite{Lazaro2008} and the references therein. For some recent progress of stochastic Hamiltonian systems with jump-type L\'evy motions (whose sample paths are càdlàg), we refer to \cite{Kolokoltsov2004,WeiChaoDuan2019,ChaoWeiDuan2021,ZhanDuan2024}. 
In additional, contact systems are getting a great popularity in recent years as they are regarded as an appropriate scenario to discuss dissipation dynamics and several other types of physical systems \cite{Grmela2014,Bravetti2017,Bravetti2019,deLeon2019JMP,Gaset2020a,Gaset2020b}. It is quite interesting and reasonable to consider the stochastic contact Hamiltonian systems (on contact manifolds), which allow to describe the mechanical systems with both dissipative forces and random noises \cite{WeiWang2021}.

We point out that even contact geometry is referred to as the ``odd-dimensional cousin” of symplectic geometry \cite{Arnold1989,Geiges2008}, the contact Hamiltonian formulations exhibit very different characteristics to their counterparts in symplectic manifolds. To solve the contact problems, one way is to consider the symplectification of the contact structure and, then, discuss the problem in the setting of homogeneous symplectic manifolds. But in general, this way is very limited. And how to deal with the problems directly in the contact settings is still a question worth thinking about. It is also worth mentioning that the contact structures are indeed Jacobi structures, more general than those of Poisson related to the symplectic ones (see Figure \ref{fig1}).
\begin{figure}[ht]
    \centering
\begin{tikzpicture}
\draw[draw=black, fill=black, fill opacity=0.1, line width=2pt]  circle [x radius=4.945, y radius=3.2];
\draw[draw=colorP, fill=colorP, fill opacity=0.1, line width=2pt] circle [x radius=2.44, y radius=2.2];
\draw[draw=colorR,fill=colorR, fill opacity=0.1, line width=2pt] (1.2,0) circle [x radius=1.1, y radius=1.15];
\draw[draw=colorB,fill=colorB, fill opacity=0.1, line width=2pt] (-1.2,0) circle [x radius=1.1, y radius=1.15];
\draw[draw=colorG,fill=colorG, fill opacity=0.1, line width=2pt] (-3.68,0) circle [x radius=1.1, y radius=1.15];
\draw[draw=colorN,fill=colorN, fill opacity=0.1, line width=2pt] (3.68,0) circle [x radius=1.1, y radius=1.15];
\draw[dashed]  (0,3.5)--(0,-3.4);
\begin{scope}
\node at (0.6, 3.4) { \emph{even}};
\node at (-0.6, 3.448) { \emph{odd}};
\node at (0, 2.7) {\large \bf Jacobi};
\node at (0, -2.7) {\large \bf $(M,\Lambda,E)$};
\node at (0, 1.5) { \bf \textcolor{colorP}{Poisson}};
\node at (0, -1.5) { \bf \textcolor{colorP}{$(M,\Lambda,0)$}};
\node at (-3.68, 0.3) {\scriptsize \bf \textcolor{colorG}{Contact}};
\node at (-3.68, -0.3) {\scriptsize \bf \textcolor{colorG}{$(M,\Lambda_\eta,-\sharp(\eta))$}};
\node at (-1.2, 0.3) {\scriptsize \bf \textcolor{colorB}{Cosymplectic}};
\node at (-1.2, -0.3) {\scriptsize \bf \textcolor{colorB}{$(M,\Lambda_{\Omega,\bar{\eta}},0)$}};
\node at (1.2, 0.3) { \scriptsize \bf \textcolor{colorR}{Symplectic}};
\node at (1.2, -0.3) {\scriptsize \bf \textcolor{colorR}{$(M,\Lambda_\omega,0)$}};
\node at (3.68, 0.3) { \scriptsize\bf \textcolor{colorN}{L.C.S.}};
\node at (3.68, -0.2) { \scriptsize \bf \textcolor{colorN}{$(M,\Lambda_{\bar{\omega}},\sharp(\theta))$}};
\node at (3.68, -0.55) { \scriptsize \bf \textcolor{colorN}{$\theta\neq 0$}};
\end{scope}
\end{tikzpicture}
\caption{Jacobi manifolds (see Table 1 in Appendix D for details of the notations).}
\label{fig1}
\end{figure}
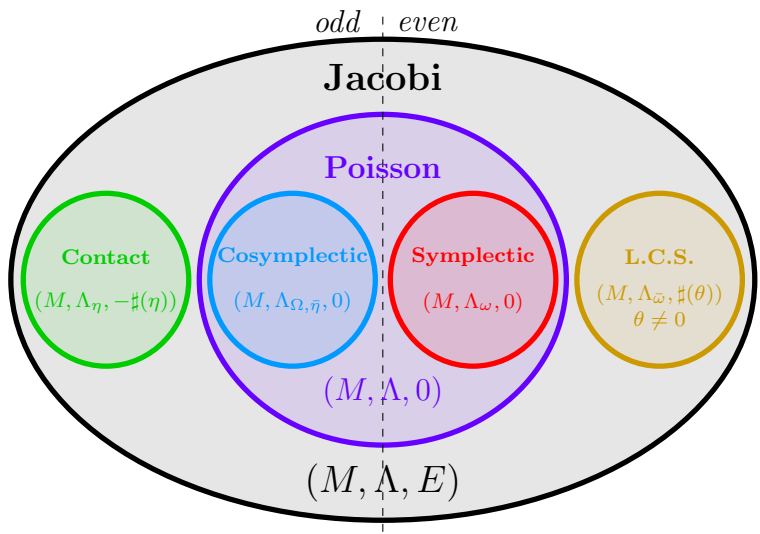
In fact, both Poisson and Jacobi manifolds were coined by Lichnerowicz \cite{Lichnerowicz1977, Lichnerowicz1978}. Compared to the famous Poisson structure, Jacobi structure has received much less attention in the physics literature until recently (The related mathematical literature is also dispersed \cite{Guedira1984,Marle1991,deLeon1997}). Besides the contact structure, there is another interesting class of Jacobi (but not Poisson) structure called locally conformal symplectic (L.C.S.) structure, that has become increasingly popular \cite{Bazzoni2018,Esen2021,Zajac2023}. The systems with L.C.S. structures usually behave like symplectic mechanical ones in some open subsets of the manifolds, although the complete global dynamics fail to be symplectic. In applications, such phenomena appear in some physical systems with nonlocal potentials or those defined by parts and each part  behaves differently and accordingly to different laws (e.g. Gaussian isokinetic dynamics with collisions and the Nos\'e--Hoover dynamics) \cite{Wojtkowski1998}.

Motivated by the above arguments, we believe that a general treatment for stochastic Hamiltonian systems on Jacobi manifolds would be useful in order to highlight the common features and remark on the differences among the many applications of this research field of growing interest. In this paper, we shall use the global stochastic analysis tools, developed originally by Meyer and Schwartz \cite{Meyer1981,Schwartz1982}, to give a strict mathematical definition of such general stochastic Hamiltonian systems. Similar to the deterministic case, the role of a vector field (i.e., the randomized Hamiltonian vector field) with a corresponding function (i.e., the randomized Hamiltonian function) with respect
to its corresponding structure, is primordial to have dynamics. And it will allow us to construct a stochastic version of the geometric Hamilton--Jacobi theory which is equivalent to stochastic Hamiltonian mechanics. 

We recall that the classical Hamilton--Jacobi theory \cite{Marsden1978,Arnold1989,Marsden1999} is designed to construct local coordinates such that the Hamilton equations expressed on these coordinates adopt a very simple form (which may be integrated). In this way, the problem of completely solving a Hamiltonian system reduces to the determination of a sufficiently large number of solutions to the Hamilton--Jacobi partial differential equation (PDE). Despite the difficulties to solve such a PDE, the Hamilton--Jacobi theory provides a remarkably powerful tool to integrate the dynamics of many Hamiltonian systems, and play an important role in classical-quantum-mechanical relationship, in numerical integration schemes (in the sense of structure preserving), in control theory or differential game theory. We also remark that this theory is originally developed by Jacobi from the variational point of view and can be also approached from the point of view of extended phase space, so that a solution to the Hamilton--Jacobi equation is referred to as the principal function or the generating function.

In the last decades,  the Hamilton--Jacobi theory has been interpreted in modern geometric terms \cite{Carinena2006a,Carinena2006b,Ferraro2017} and has been extended to contact systems \cite{deLeon2017,deLeon2021contact,deLeon2023contact}, L.C.S. systems \cite{Esen2021,Zajac2023} and many other different contexts (e.g., nonholonomic systems, singular Lagrangian systems, higher-order systems or field theories) \cite{Ohsawa2009Nonholonomic,deLeon2013singularL,deLeon2020field,deLeon2022geometric}. On the other hand, stochastic versions of Hamilton--Jacobi equations (usually referred to as stochastic Hamilton--Jacobi--Bellman equations) have drawn more and more attention as they are useful when studying various classes of stochastic models in probability theory and mathematical physics \cite{Bismut1980,Kolokoltsov1996,Souganidis1999,Rezakhanlou2000,Kolokoltsov2004,LazaroCam2009}. For example, they can describe the evolution of optimally controlled systems with random dynamics \cite{Kolokoltsov1998,Mikami2021} and are expected to be further applied to in some fields like AI or reinforcement learning \cite{Peyre2019,HuangZambrini1}.

The content of the paper is as follows. Section 2 is devoted to introducing the main ingredients of Jacobi manifold as well as the interpretation of a Hamiltonian system on Jacobi manifold. In Section 3, we formulate the stochastic Hamiltonian systems on Jacobi manifold, and show that the solution flows of such systems preserve the characteristic leaves and the structures of these leaves. In Section 4, we develop a Hamilton--Jacobi (HJ) theory for stochastic Hamiltonian systems on both contact manifold and L.C.S. manifold. 


\section{Preliminaries}
In this preliminary section, we shall review some basic facts about Hamiltonian systems on Jacobi manifold as well as that about stochastic calculus on general smooth manifolds.

\subsection{Jacobi structures \& manifolds}
\begin{definition}
A \textit{Jacobi structure} \cite{Lichnerowicz1978,deLeon1997} on an $m$-dimensional (smooth) manifold $M$ is a pair $(\Lambda, E)$, where $\Lambda\in\mathcal{T}^{2}(M)$ is a bivector field (a skew-symmetric contravariant 2-tensor field) and $E\in \mathfrak{X}(M)$ is a vector field, so that the following identities are satisfied:
\begin{equation}
[\Lambda, \Lambda]_{SN} = 2E\wedge \Lambda, \quad \; \mathscr{L}_E\Lambda = [E, \Lambda]_{SN} = 0,
\end{equation}
where $[\cdot, \cdot]_{SN}$ is the Schouten--Nijenhuis bracket and $\mathscr{L}$ is the Lie derivative operator. The manifold $M$ endowed with a Jacobi structure is called a \textit{Jacobi manifold}, and it is also often referred to as a triple $(M,\Lambda, E)$. 
\end{definition}
Note that the Jacobi structure $(\Lambda, E)$ induces a skew-symmetric bilinear map on the space of smooth functions $C^\infty(M)$, called the \textit{Jacobi bracket} and given by
\begin{equation}\label{Jacobi bracket}
  \{f,g\} = \Lambda(df,dg)+fE(g)-gE(f), \quad  \forall f,g\in C^\infty(M). 
\end{equation}
It satisfies the Jacobi identity and the so-called weak Leibniz rule:
$\mathrm{supp} (\{ f,g\}) \subseteq \mathrm{supp} (f) \cap \mathrm{supp} (g)$. 
It can be regarded as a generalization of the usual Poisson bracket on $C^\infty(M)$, since the only difference is to replace the Leibniz rule by 
$\{f,gh\}=g \{f,h\}+h \{f,g\}+ghE(f)$. 
The space $(C^\infty(M), \{ \cdot,\cdot\})$ is called a local Lie algebra by Kirillov \cite{Kirillov1976}. Conversely, a local Lie algebra on $C^\infty(M)$ defines a Jacobi structure on $M$ \cite{Kirillov1976,Guedira1984}. 
If the vector field $E$ identically vanishes, then the manifold $M$ reduces to a Poisson manifold, and the corresponding bracket $\{ \cdot,\cdot\}$ is a Poisson bracket which is a derivation in each argument.

Clearly, symplectic (as well as cosymplectic, and  Lie--Possion) structures can be regarded as Jacobi structures, as they are all important  examples of Poisson ones. Contact and L.C.S. structures are also Jacobi but they are not Poisson. See Appendix (A-C) for more information on these special Jacobi cases. We also refer to Gu\'edira, Lichnerowicz $\&$ Marle \cite{Dazord1991}, Marle \cite{Marle1991}, de L\'eon, Marrero $\&$ Padr\'on \cite{deLeon1997} for general properties of Jacobi manifolds.

\subsection{Hamiltonian formulations}

The Jacobi structure induces a morphism between the $C^\infty(M)$-module of one-forms $\Omega^1(M)$ and that of vector fields $\mathfrak{X}(M)$, and it is natural to define a (Jacobi-type) Hamiltonian vector field. 

\begin{definition}\label{def-hvf}
Let $(M,\Lambda, E)$ be a Jacobi manifold. Define a mapping by 
\begin{equation}\label{Sharp-Lambda}
\Lambda^\sharp : \Omega^1(M)\to \mathfrak{X}(M), \quad  \alpha \mapsto \Lambda(\alpha,\cdot),
\end{equation}
that is, $\Lambda^\sharp(\alpha)(\beta)=\Lambda(\alpha,\beta)$, for any $\alpha,\beta\in \Omega^1(M)$. Given a $C^\infty$ real-valued function $f$ on $M$ which is called a \textit{Hamiltonian function}, 
the vector field $V_f$ defined by
\begin{equation}\label{Hamiltonian-vf}
V_f=\Lambda^\sharp(df)+fE
\end{equation}
is called the \textit{Hamiltonian vector field} associated with $f$. 
\end{definition}

The Jacobi bracket $\{ \cdot, \cdot \}$ defined in \eqref{Jacobi bracket} can be related with the Hamiltonian vector fields by
\begin{equation}\label{Jacobi-Ham-vec}
  \{f,g\} = V_f(g) -gE(f), \quad  \forall f,g\in C^\infty(M).
\end{equation}
It should be noticed that the Hamiltonian vector field associated with the constant function $f\equiv 1$ is just $E$. A direct computation shows that $[V_f,V_g]_{SN}=X_{\{f,g\}}$, $\forall f,g\in C^\infty(M)$ \cite{Lichnerowicz1978,deLeon1997}.

An integral curve $\gamma$ of the Hamiltonian vector field $V_f$ satisfies the so-called \textit{Hamiltonian equation}: $\dot{\gamma}(t) =V_f(\gamma(t))$. To generalize the definition of Hamiltonian equations on $(M,\Lambda,E)$ to the stochastic context, we need the following preparation.

\begin{proposition}\label{along-curve}
Let $(M,\Lambda,E)$ be a Jacobi manifold and $f\in C^\infty (M)$. The smooth curve $\gamma: [0,T] \to M$ is an integral curve of the Hamiltonian vector field $V_f$ if and only if for any $\alpha\in\Omega^1(M)$ and for any
$t \in [0,T]$,
\begin{equation}\label{JH1}
  \alpha(\dot{\gamma}(t)) = - df(\Lambda^\sharp(\alpha)) (\gamma(t)) + f\alpha(E) (\gamma(t)),
\end{equation}
or equivalently,
\begin{equation}\label{JH2}
  \int_{\gamma|_{[0,t]}} \alpha =\int_0^t \left[-df(\Lambda^\sharp(\alpha))+f \alpha(E)\right](\gamma(s))ds.
\end{equation}
\end{proposition}

\begin{proof}
We remark that a proof of the Poisson case can be found in \cite{Lazaro2008}. Here, for the Jacobi case, by taking \eqref{Sharp-Lambda} and \eqref{Hamiltonian-vf} into account, we conclude that $\dot{\gamma}=V_f(\gamma)$ is equivalent to the following expression:
\begin{align*}
\alpha(\dot{\gamma}) &= \alpha(V_f) (\gamma) = \alpha(\Lambda^\sharp(df) + fE) (\gamma) \\
&= -\Lambda(\alpha, df) (\gamma) + f\alpha(E) (\gamma) \\
&=- df(\Lambda^\sharp(\alpha)) (\gamma) + f\alpha(E) (\gamma).
\end{align*}
for any $\alpha\in\Omega^1(M)$. This proves \eqref{JH1}, and \eqref{JH2} thus follows by taking integrals with respect to $t$ on both sides of \eqref{JH1}. 
\end{proof}

There are two particular cases:
\begin{itemize}
    \item[(i)] If $M$ is a contact manifold with the contact 1-form $\eta$, then
 \begin{equation}\label{contact-case}
E=-\mathcal{R},\quad \Lambda^\sharp(\alpha)=\sharp(\alpha)-\alpha(\mathcal{R})\mathcal{R},
 \end{equation}
where $\mathcal{R}$ is the Reeb vector field, $\sharp=\flat^{-1}$ and $\flat: TM \to T^\ast M ,\; V \mapsto \flat(V)=\iota_V d\eta +\eta(V)\eta$;
\item[(ii)] If $M$ is an (almost) symplectic manifold with the (almost) symplectic 2-form $\omega$, then 
 \begin{equation}\label{symplectic-case}
E=0,\quad \Lambda^\sharp(\alpha)=\tilde{\sharp}(\alpha),
 \end{equation}
 where $\tilde{\sharp}=\tilde{\flat}^{-1}$ and $\tilde{\flat}: TM \to T^\ast M ,\; V \mapsto \tilde{\flat}(V)=\iota_V\omega$. 
\end{itemize}

   Proposition \ref{along-curve} inspires us to define the Hamiltonian equations by specifying the result of integrating an arbitrary one form $\alpha\in\Omega^1(M)$ along them. In contrast to \eqref{JH2}, we will see that, in \eqref{SJHE-integral}, the processes solving stochastic Hamiltonian equations are no longer just driven by the deterministic time, but also the stochastic component.  

\subsection{Characteristic foliations and Lagrangian--Legendrian submanifolds}

We next introduce the characteristic distribution as well as the corresponding foliation of a Jacobi manifold. For each $z\in M$, we consider the subspace $\mathcal{C}_z\subset T_zM$ generated by the values of all Hamiltonian vector fields. That is, $\mathcal{C}_z$ is generated by the image of the linear map ${\Lambda^\sharp}_z: T_z^\ast M \to T_zM$ and the vector $E_z$ in the form of
\begin{equation}\label{characteristic distribution}
\mathcal{C}_z={\Lambda^\sharp}_z(T_z^\ast M) \oplus <E_z>.
\end{equation}
In this way, $\mathcal{C}=\cup_{z\in M}\mathcal{C}_z\subset TM$ is referred to as the \textit{characteristic distribution} of $(M,\Lambda, E)$. The Jacobi manifold is said to be \textit{transitive} if $\mathcal{C}=TM$.

\begin{lemma}\label{foliation} \cite{Kirillov1976,Marle1991}
The characteristic distribution $\mathcal{C}$ is completely integrable in the sense of Stefan--Sussmann, thus defines on $M$ a Stefan foliation $\mathcal{M}_F$ (i.e., a foliation whose leaves are not necessarily of the same dimension), called the characteristic foliation. Furthermore, the following properties hold:
\begin{itemize}
    \item[(i)] Each leaf $\mathcal{L}$ of $\mathcal{M}_F$ has a unique transitive Jacobi structure $(\mathcal{L},\Lambda_L, E_L)$ such that its canonical injection into $M$ (i.e., $i:\mathcal{L} \hookrightarrow M$) is a Jacobi map in the sense of $\{f\circ i,g\circ i\}_{\mathcal{L}}=\{f,g\}\circ i$, $\forall f,g\in C^\infty(M)$. 
    \item[(ii)]  Each leaf $\mathcal{L}$ of $\mathcal{M}_F$ can be either a contact manifold (with the induced Jacobi structure) if the dimension is odd, or an L.C.S. manifold if the dimension is even. In particular, for a Poisson manifold (i.e., $E=0$), each leaf is a symplectic manifold (i.e., an L.C.S. manifold with the Lee form $\gamma=0$).
\end{itemize}
\end{lemma}

Note that the construction of a Hamilton--Jacobi theory often relies in the existence of a special submanifold \cite{Tulczyjew1976a,Tulczyjew1976b}. We also need to consider the following submanifold of a Jacobi manifold.

\begin{definition}
    A submanifold $N\subset M$ of the Jacobi manifold is said to be a \textit{Lagrangian--Legendrian submanifold} if the following equality holds
\begin{equation}
\Lambda^\sharp(TN)^{\circ}=TN\cap \mathcal{C},
\end{equation}
where $(TN)^{\circ}$ denotes the annihilator of $TN$, i.e., $(T_zN)^{\circ}=\{\alpha_z\in T_z^\ast M:\;\alpha_z(T_zN)=0\}$, for $z\in M$.
\end{definition}

 For the Poisson case, the submanifold $N$ coincides with the classical \textit{Lagrangian} one. For the contact case (of dimension $2n+1$), the submanifold $N$ is just called \textit{Legendrian}, and it is indeed an integral manifold of maximal dimension $n$ of the distribution \cite{deLeon2017}.

\subsection{Stochastic calculus on manifolds}

The fact that we are dealing with stochastic systems also leads to more notations and bookkeepings. Here we review some basic concepts in stochastic calculus on manifolds, and we refer to \'Emery \cite{Emery1989}, Hsu \cite{Hsu2002} and Gliklikh \cite{Gliklikh2011} for more details.

We start with an $m$-dimensional smooth manifold $M$ and a filtered probability space $(\mathbf\Omega, \mathscr{F}, \{ \mathscr{F}_t \}_{t\geqslant0}, \mathbb{P})$. Recall that an $\{ \mathscr{F}_t \}$-adapted process $X:[0,\tau)\times {\mathbf\Omega}\to M $ is called an \textit{$M$-valued semimartingale}, if $f(X)$ is a real-valued semimartingale on $[0,\tau)$ for all $f\in C^\infty(M)$, where $\tau$ is an $\{ \mathscr{F}_t \}$-stopping time named the lifetime of $X$. The process $X$ can be extended to the half real line $[0,\infty)$, if we apply the one-point compactification of $M$ as $\overline M := M \cup \{\pt_M\}$, and set $X(t)=\partial_M$ for all $t \geq \tau$. We denote by $[\cdot,\cdot]$ the quadratic covariation for semimartingales on Euclidean spaces.

Stochastic differential equations (SDEs) on a manifold are usually defined by some given vector fields, driving semimartingales and initial random variables. There are different types of SDEs. The most natural way in our context is to introduce the Stratonovich SDEs, and another remarkable description is based on It\^o stochastic differentials. 
\begin{definition}
    For an $M$-valued semimartingale $X$, its \textit{Stratonovich stochastic differential} $\delta X_t$ at time $t$ is defined as the $T_{X_t} M$-valued stochastic differential, or the stochastic differential-valued vector at $X_t$, with local coordinate expression
\begin{equation*}
  \delta X_t = \delta X_t^i \frac{\pt}{\pt x^i}\bigg|_{X_t}.
\end{equation*}
The \textit{It\^o stochastic differential} $\d X_t$ at time $t$ is defined as the $\mathcal{T}^S_{X_t} M$-valued stochastic differential, or the stochastic differential-valued vector at $X_t$, with local coordinate expression
\begin{equation*}
  \d X_t = d X_t^i \frac{\pt}{\pt x^i}\bigg|_{X_t} + \frac{1}{2} d\left[ X^j, X^k \right]_t \frac{\pt^2}{\pt x^j\pt x^k}\bigg|_{X_t}.
\end{equation*}
Here, $TM$ and $\mathcal{T}^SM$ stand for (first-order) tangent bundle and second-order tangent bundle of $M$, respectively.
\end{definition}

We remark that Stratonovich differentials obey the ordinary (first-order) differential calculus, and we may obtain Stratonovich differentials as limits of ordinary (deterministic) ones when the semimartingale $X$ is suitably approximated by a piecewise smooth curve. This leads to the convenient fact that many classical geometrical constructions involving smooth curves extend to semimartingales via Stratonovich differentials (this is called the \textit{transfer principle} by Malliavin). However, It\^o differentials are not invariant under changes of coordinates, and a change of coordinates always produces an additional term owing to It\^o's formula. The advantage of It\^o differentials mainly lies in analyzing martingale and stochastic properties. For more information on these two kinds of stochastic differentials, see \cite[Chapter 7]{Emery1989} and \cite[Section 7.2]{Gliklikh2011}.



\section{Stochastic Hamiltonian Systems (SHSs) on Jacobi Manifolds}

In this section, we adopt the Schwartz--Meyer approach, combined with Bismut's perspective, to define stochastic Hamiltonian systems on Jacobi manifolds. Let $(M,\Lambda,E)$ be an $m$-dimensional Jacobi manifold.

\subsection{Definitions of SHSs in Stratonovich and It\^o forms}

We consider an $\mathbb{R}^r$-valued continuous semimartingale $X:\mathbb{R}_+\times \mathbf\Omega \to \mathbb{R}^{r}$. We denote by $\cdot$ the inner product in Euclidean spaces. Let $h = (h_1,\cdots,h_r)$ be a smooth map from $M$ to $\mathbb{R}^r$.

\begin{definition}\label{Def-SHS}(SHSs in Stratonovich form)
  A stochastic Hamiltonian system driven by the semimartingale $X$ is of the form of the following Stratonovich SDE
\begin{equation}\label{SJHE}
\delta Z_t^h=\mathcal{H}(X_t,Z_t^h) \delta  X_t, 
\end{equation}
where the symbol $\delta $ stands for the Stratonovich differential operator, the linear operator $\mathcal{H}(x, z) : T_x \mathbb{R}^{r} \cong \mathbb{R}^{r} \to T_z M$, for each $x\in\R^r$ and $z\in M$, is defined by 
\begin{equation}\label{SJH}
\mathcal{H}(x, z)(v)=\sum_{k=1}^r v^k V_{h_k}(z) = V_{v\cdot h}(z),
\end{equation}
and referred to as a Hamiltonian Stratonovich operator. Let $Z^h:[0,\tau^h)\times {\mathbf\Omega}\to M$ be a solution of \eqref{SJHE} with lifetime $\tau^h:{\mathbf\Omega}\to \mathbb{R}_+$. We call the pair $(Z^h, \tau^h)$ a \textit{(Jacobi-type) Hamiltonian semimartingale}; we will omit the lifetime $\tau^h$ when we do not emphasize it.
\end{definition}

According to \cite[Definition (7.16)]{Emery1989}, $Z^h:[0,\tau^h)\times {\mathbf\Omega}\to M$ is a solution of \eqref{SJHE} if and only if the following Stratonovich integral equation holds:
\begin{equation}\label{SJHE-integral}
\int \alpha \left( \delta Z_t^h \right) =\int \mathcal{H}^\ast(X_t, Z_t^h)(\alpha) (\delta  X_t), \quad \forall\alpha\in\Omega^1(M),
\end{equation}
where $\mathcal{H}^\ast(x, z): T_z^\ast M \to T_x^\ast \mathbb{R}^{r} \cong \mathbb{R}^{r}$ is the dual operator of $\mathcal{H}(x, z)$, which is, thanks to Proposition \ref{along-curve}, given by that for any 1-form $\alpha$ on $M$ and $u\in\R^r$,
\begin{equation}\label{dual-Stratonovich}
  \begin{split}
    \mathcal{H}^\ast(x, z)(\alpha) (u) &= \alpha (\mathcal{H}(x, z)(u)) = \alpha (V_{u\cdot h})(z) = \Lambda(d(u\cdot h), \alpha) (z) + u\cdot h(z) \alpha(E) (z) \\
    &= \sum_{k=1}^r u^k \left[ -dh_k(\Lambda^\sharp(\alpha)) (z) + h_k(z) \alpha(E)(z) \right].
  \end{split}
\end{equation}

\begin{remark}
    This definition can be regarded as a Jacobi-version of Definition 2.1 in L\'azaro-Cam\'i $\&$ Ortega \cite{Lazaro2008}. According to Proposition \ref{along-curve}, equation \eqref{SJHE} reduces to a standard stochastic Hamiltonian equation on a symplectic manifold (see, e.g., \cite{Bismut1981,Lazaro2008,WeiChaoDuan2019}), by taking $E=0$ and $\Lambda^\sharp=\tilde{\flat}^{-1}(\alpha)$; And it is referred to as a stochastic contact Hamiltonian equation (see, e.g.,\cite{WeiWang2021}) if we take $E=-\mathcal{R}$ and $\Lambda^\sharp(\alpha)=\sharp(\alpha)-\alpha(\mathcal{R})\mathcal{R}$. We note that the standard symbol (namely ``$\omega$") for the symplectic form is, unfortunately, also commonly used for the sample-space chance variable. However, this overlap in notation should not cause confusion, given the usual convention of suppressing the chance variable in SDEs.
\end{remark}

\begin{remark}
  The driving semimartingale $X$ can be considered as a more general one taking values on a (finite-dimensional) vector space (cf. the Poisson case in \cite{Lazaro2008}). The continuous condition of $X$ can be also weakened to be c\`{a}dl\`{a}g (that is, right-continuous with left limit at each time instant, a.s.), and,  accordingly, the Stratonovich differential in \eqref{SJHE} should be replaced by the more general Marcus differential \cite{Applebaum2009,WeiChaoDuan2019}.
\end{remark}

From a physical point of view, the stochastic system given by \eqref{SJHE} can be regarded as a Hamiltonian one within the external world \cite{Bismut1981,Misawa1999}. In fact, by rewriting \eqref{SJHE} into a more familiar vector field form:
\begin{equation}\label{SJHE3}
  \delta Z_t^h = \sum_{k=1}^r V_{h_k}(Z_t^h) \delta  X_t^k = V_{h \cdot\delta X_t} (Z_t^h),
\end{equation}
we can perceive that the stochastic part is introduced to characterize the complicated interaction between the ``deterministic" Hamiltonian system (with the Hamiltonian function $h_0$) and the fluctuating environment. Formally, equation \eqref{SJHE3} can be regarded as a generalized Hamiltonian system with a ``randomized'' Hamiltonian $
h \cdot\delta X_t = \sum_{k=1}^r h_k \delta  X_t^k.
$ 

We shall also consider the special case that $X_t = (t, B_t)$ where $B$ is a standard Brownian motion. By a re-indexing for convenience, \eqref{SJHE3} reads
\begin{equation}\label{SJHE-BM}
  \delta Z_t^h=\mathcal{H}(t, B_t,Z_t^h) (dt, \delta  B_t) = V_{h_0}(Z_t^h) dt + \sum_{k=1}^r V_{h_k}(Z_t^h) \delta  B_t^k = V_{h_0 dt + \sum_{k=1}^r h_k \delta  B_t^k} (Z_t^h).
\end{equation}

\begin{remark}
    Even though the deterministic Hamiltonian structure has been ``destroyed" due to the existence of noise, the whole stochastic system may still exhibit some geometric properties; see Section \ref{Fundamental}. Furthermore, by introducing certain integrability conditions and a small parameter $\varepsilon$, the formula for $h_0 dt + \varepsilon \sum_{k=1}^r h_k \delta  B_t^k$ is quite similar to that of the famous nearly integrable systems (e.g., with  $H=h_0+\varepsilon h_1$). This may lead to some interesting problems, e.g., the long-time behaviors and averaging principles \cite{Arnold1989,Lixm2008,Freidlin2012,WeiChaoDuan2019}.
\end{remark}

\begin{remark}
    We emphasize that directly using the random vector field $V_{h \cdot\delta X_t}$ and the random function $h \cdot\delta X_t$ is not entirely rigorous, as they involve generalized differentials of the noise. To handle them properly, one should rely on the stochastic equations \eqref{SJHE} or \eqref{SJHE-integral}. Nevertheless, the idea of “randomization” still proves useful for streamlining certain parts of the auxiliary analyses and proofs. Indeed, by adopting these randomized forms, one can more readily (at least at a formal level) obtain stochastic analogues of many features from the deterministic setting.
\end{remark}

In the view of stochastic analysis, the Stratonovich differential $\delta  X_t$ in \eqref{SJHE} can be regarded as a first-order tangent vector, and the mapping $\mathcal{H}$ is just a section of the vector bundle $T^\ast \R^r \otimes TM$ over the base $\R^r\times M$. It follows from \'Emery \cite[Lemma (7.22)]{Emery1989} that, for the Hamiltonian Stratonovich operator $\mathcal{H}(x,z)$, there exists a unique Schwartz operator $\widetilde{\mathcal{H}}(x, z) : \tau_x \mathbb{R}^r \cong \R^r \times \mathrm{Sym}^2(\R^r) \to \tau_z M$ such that for each smooth curve $(x(t), z(t))\in \R^d \times M$ verifying $\mathcal{H}(x(t), z(t)) \dot x(t) = \dot z(t)$ for all $t$, one has also $\widetilde{\mathcal{H}}(x(t), z(t))(\ddot x(t)) = \ddot{z}(t)$, where $\ddot x(t)\in \tau_{x(t)} N$ is the second-order vector associated to the ``acceleration'' of the curve $(x(t))$ \cite[Example (6.5)]{Emery1989}, i.e., $\ddot x(t) (f) = (f\circ x)''(t)$ for any smooth $f$ on $M$; so is $\ddot z(t)$. In other words, $\widetilde{\mathcal{H}}$ is a section of the vector bundle $\tau^\ast \R^r\otimes \tau M$ over $\R^r\times M$. We call $\widetilde{\mathcal{H}}$ the Hamiltonian Schwartz operator associated with $\mathcal{H}$. It gives rise to the following equivalent definition for SHSs:

\begin{lemma}\label{Def-SHS-Ito} (SHSs in It\^o form) The stochastic Hamiltonian system \eqref{SJHE} in Stratonovich form is equivalent to the following It\^o SDE:
\begin{equation}\label{SHS-Ito}
  \d Z_t^h = \widetilde{\mathcal{H}}(X_t,Z_t^h)\d X_t,
\end{equation}
where $\d$ denotes the It\^o differential operator valued on $\Gamma(\tau M)$; equivalently, for every second-order form $\theta$ on $M$,
\begin{equation*}
  \int \theta \left(\d Z_t^h\right) = \widetilde{\mathcal{H}}^\ast(X_t,Z_t^h)(\theta) \left(\d X_t\right),
\end{equation*}
where $\widetilde{\mathcal{H}}^\ast(x, z) : \tau^*_z M \to \tau^*_x \mathbb{R}^r$ is the dual operator of $\widetilde{\mathcal{H}}(x, z)$.
\end{lemma}



Recall the Jacobi bracket $\{\cdot,\cdot\}$ defined in \eqref{Jacobi bracket}. We denote by $d^2$ the second-order differential operator on $M$.

\begin{proposition}\label{Jacobi bracket form}
  For any $f\in C^\infty(M)$, $u\in\R^r$ and $a = (a^k, a^{kl})\in \R^r \times \mathrm{Sym}^2(\R^r)$, we have
\begin{align}\label{Str-op}
  \mathcal{H}^\ast(x, z)(df)(u) = \sum_{k=1}^r u^k \big(\{h_k,f\}+fE(h_k) \big)(z),
\end{align}
  and 
\begin{align}\label{Swa-op}
  \widetilde{\mathcal{H}}^\ast(x, z) \left(d^2f\right) (a) = \sum_{k,l=0}^r a^k \big(\{h_k,f\} + fE(h_k) \big)(z) + a^{kl} C_{k,l}^f(z),
\end{align}
where 
$C_{k,l}^f(z) = \big(\{h_k,\{h_l,f\}\}+\{h_k,fE(h_l)\}+\{h_l,f\}E(h_k)+f E(h_l)E(h_k)\big)(z)$.
Furthermore, the Hamiltonian semimartingale $Z^h$ fulfills
\begin{align}
f(Z_t^h)-f(Z_0^h)
=&\sum_{k=1}^r \int_0^t \big(\{h_k,f\}+fE(h_k) \big)(Z_s^h) \left( \delta  X_s^k \right) \label{fZ}\\
=&\sum_{k,l=0}^r \int_0^t\big(\{h_k,f\}+fE(h_k) \big)(Z_s^h) \left( d X_s^k \right) + C_{k,l}^f(Z_s^h) d\left[ X^k, X^l \right]_s,
\label{SHS-Ito-fZ}
\end{align}
where the last term of \eqref{SHS-Ito-fZ} is callded the Stratonovitch (or Wong--Zakai) correction.
\end{proposition}

\begin{proof}
Taking $\alpha=df$ in \eqref{dual-Stratonovich} and using \eqref{Jacobi-Ham-vec},
\begin{equation}\label{Schwartz operator}
  \mathcal{H}^\ast(x, z)(df)(u) = df \left( \mathcal{H}(x, z)(u) \right) = df (V_{u\cdot h}(z)) = \left[ \{u\cdot h , f\} + f E(u\cdot h) \right] (z).
\end{equation}
Equation \eqref{Str-op} follows. To obtain \eqref{Swa-op}, we fix an $a = (a^k, a^{kl})\in \tau_x \R^r \cong \R^r \times \mathrm{Sym}^2(\R^r)$. Let $x(t)$ be a smooth curve on $\R^r$ such that $x(0) = x$ and $\ddot x(0) = a$, that is, $\dot x^k(0) = a^k$ and $\ddot x^{kl}(0) = a^{kl}$. Let $z(t)$ be a smooth curve on $M$ such that $z(0) = z$. We then calculate that 
\begin{align}
  &\quad \widetilde{\mathcal{H}}^\ast(x,z) \left(d^2 f\right) (a) \notag \\
  &= d^2 f \left( \widetilde{\mathcal{H}}(x,z) (a)\right) = d^2 f \left( \ddot z(0) \right) = (f\circ z)''(0) \notag \\
  &= \frac{d}{dt}\bigg|_{t=0} df ( \mathcal{H}(x(t), z(t)) (\dot x(t)) ) \notag \\
  &= \frac{d}{dt}\bigg|_{t=0}  \big( \{\langle \dot x(t), h \rangle , f\} + f E(\langle \dot x(t), h \rangle) \big) (z(t)) \notag \\
  &= \big( \{\langle \ddot x(0), h \rangle , f\} + f E(\langle \ddot x(0), h \rangle) \big) (z) + \dot z(0) \big( \{\langle \dot x(0), h \rangle , f\} + f E(\langle \dot x(0), h \rangle) \big) \notag \\
  &= \big( \{\langle \ddot x(0), h \rangle , f\} + f E(\langle \ddot x(0), h \rangle) \big) (z)+ \big[ \{\langle \dot x(0), h \rangle , \{\langle \dot x(0), h \rangle , f\} + f E(\langle \dot x(0), h \rangle)\} \notag \\
  &\quad  + (\{\langle \dot x(0), h \rangle , f\} + f E(\langle \dot x(0), h \rangle) ) E(\langle \dot x(0), h \rangle) \big] (z) \notag \\
  &= a^k \big( \{h_k, f\} + f E(h_k) \big) (z) \notag \\
  &\quad + a^{kl} \big( \{ h_k, \{h_l, f\} + f E(h_l)\} + (\{h_l, f\} + f E(h_l)) E(h_k) \big) (z), \notag
\end{align}
which proves \eqref{Swa-op}. We conclude the results in \eqref{fZ} and \eqref{SHS-Ito-fZ} by the following fact due to \cite[Theorem (6.24) and Proposition (7.4)]{Emery1989}:
\begin{equation*}
  \begin{split}
    f(Z_t^h)-f(Z_0^h) &= \int_0^t df\left( \delta  Z_s^h \right) = \int_0^t \mathcal{H}^\ast (X_s,Z_s^h)(d f) (\delta  X_s) \\
    &= \int_0^t d^2f\left( \d Z_s^h \right) = \int_0^t \widetilde{\mathcal{H}}^\ast (X_s,Z_s^h)\left(d^2 f\right) (\d X_s).
  \end{split}
\end{equation*}
The proof is completed.
\end{proof}

\subsection{Preservation of characteristic leaves $\&$ structures}\label{Fundamental}

The existence and uniqueness of the solution to the stochastic Hamiltonian system \eqref{SJHE} or \eqref{SHS-Ito} can be ensured (up to a maximal stopping time), referring to \cite[Theorems 7.21 and 6.41]{Emery1989}. We note that the well-posedness of such equations can be global in time if we 
focus on some special cases like Langevin systems \cite{Song2020}; see also \cite{Hsu2002,IW89,Gliklikh2011} for more on this topic. 

We now denote by $\varphi=\{\varphi_t(z,\omega): z\in M, 0\leq t <\tau^h(\omega), \omega\in \mathbf\Omega\}$ the solution flow associated with the Hamiltonian
semimartingale $(Z^h, \tau^h)$, and call it a \textit{stochastic Hamiltonian flow}. Thanks to \cite[Section V.2]{IW89}, $\varphi_t(\cdot,\omega):M \to M$ is a local $C^1$ diffeomorphism for $t\ge0$ and almost all $\omega\in \mathbf\Omega$ such that $t<\tau^h(\omega)$. In addition, such a stochastic flow acts naturally on tensor fields and particularly on differential forms, according to \cite[Section 3]{Kunita1981} and \cite[Chapter 7]{Kunita2019}.

Recall that the stochastic Hamiltonian system can be rewritten into the Stratonovich form \eqref{SJHE3}, same as the SDE considered in \cite{Kunita1981}. The following lemma is adapted from \cite[Theorem 3.3]{Kunita1981}.

\begin{lemma}\label{Lie derivative form}
  Let $\varphi$ be the stochastic Hamiltonian flow associated with the Hamiltonian semimartingale $(Z^h, \tau^h)$. For a form $\alpha$ on $M$ of arbitrary order, it holds almost surely for $t<\tau^h$ that
  \begin{align}
    \varphi_t^\ast \alpha 
    &= \alpha +\sum_{k=1}^r\int_0^t \varphi_s^\ast \mathscr{L}_{V_{h_k}}\alpha \delta  X_s^k \label{arbitrary-form}\\
    &= \alpha +\sum_{k=1}^r\int_0^t  \varphi_s^\ast \mathscr{L}_{V_{h_k}}\alpha \d X_s^k 
    + \frac{1}{2} \sum_{k,l=1}^d \int_0^t  \varphi_s^\ast \mathscr{L}_{V_{h_k}}\mathscr{L}_{V_{h_l}}\alpha d\left[X^l, X^k \right]_s, \label{arbitrary-form-Ito}
  \end{align}
  where $\varphi_t^\ast$ denotes the pull-back of $\varphi_t$.
\end{lemma}

Denote by $\mathscr{L}_{V_{h \cdot\delta X_t}}\alpha$ the Lie derivative of $\alpha$ with respect to the randomized Hamiltonian vector field $V_{h \cdot\delta X_t}$ given in \eqref{SJHE3}. Geometrically, it describes the change of the differential form $\alpha$ in the direction of the stochastic Hamiltonian flow generated by the vector field $V_{h \cdot\delta X_t}$. Then, we can further rewrite the equation above into
\begin{equation}\label{arbitrary-form2}
  d \varphi_t^\ast \alpha = \varphi_t^\ast \mathscr{L}_{V_{h \cdot\delta X_t}}\alpha.
\end{equation}

\begin{theorem} \label{characteristic foliation} (Preservation of characteristic leaves $\&$ structures) 
Given a characteristic leaf $\mathcal{L}$ of $(M,\Lambda,E)$. 
Let $\varphi$ be the stochastic Hamiltonian flow associated with the Hamiltonian semimartingale $(Z^h, \tau^h)$. 
If the initial state $Z_0^h\in \mathcal{L}$ a.s., then for all $t<\tau^h$, we have, a.s.:
(i)
$\varphi_t\in \mathcal{L}$; (ii) $\varphi_t$ preserves the contact structure of $\mathcal{L}$ if $\mathrm{dim}(\mathcal{L})$ is odd, or the L.C.S. structure if $\mathrm{dim}(\mathcal{L})$ is even.
\end{theorem} 

\begin{proof}
($i$) We follow the lines of \cite[Proposition 2.4]{LazaroCam2009}. By the definition of Stratonovich operator in \eqref{SJH}, for any $z\in \mathcal{L}$, $\mathcal{H}(x, z)$ takes values in the tangent space $T\mathcal{L}$ of $\mathcal{L}$. It thus induces a new Stratonovich operator $\mathcal{H}_\mathcal{L}(x, z): T_x \R^r\to T\mathcal{L}$ by restricting $\mathcal{H}$ to its range. For the inclusion $i:\mathcal{L} \hookrightarrow M$, we have 
$\mathcal{H}_\mathcal{L}^\ast (x, z)\circ T_z^\ast i=\mathcal{H}^\ast(x, i(z)).$ To prove the statement, we only need to show that: if $Z_\mathcal{L}^h$ is a solution of the following Stratonovich SDE
\begin{align}
dZ_\mathcal{L}^h=\mathcal{H}_\mathcal{L}(X,Z_\mathcal{L}^h) (\delta  X),
\end{align}
then $\hat{Z}^h=i\circ Z_\mathcal{L}^h$ is a solution to \eqref{SJHE}. In fact, for any $\alpha\in \Omega(M)$,
\begin{equation}
  \int \alpha \left( d\hat{Z}^h \right)
  =\int \alpha \left( d(i\circ Z_\mathcal{L}^h) \right)
  =\int T^\ast i (\alpha) \left( dZ_\mathcal{L}^h \right).
\end{equation}
Since $T^\ast i (\alpha)\in \Omega(\mathcal{L})$, the above equation further equals
\begin{align}
  \int \mathcal{H}_\mathcal{L}^\ast (X, Z_\mathcal{L}^h) T^\ast i (\alpha)(\delta  X)
  =&\int \mathcal{H}^\ast (X, i\circ Z_\mathcal{L}^h)(\alpha)(\delta X) \notag \\
  =&\int \mathcal{H}^\ast (X, \hat{Z}^h)(\alpha)(\delta  X).
\end{align}
Consequently, $d \hat{Z}^h=\mathcal{H}(X,\hat{Z}^h) (\delta  X)$ and the result follows by the uniqueness of the solution to an SDE.

($ii$) By Lemma \ref{foliation}, if $\mathrm{dim}(\mathcal{L})$ is odd, then $\mathcal{L}$ is a manifold with a contact 1-form $\eta_\mathcal{L}$. In this case, we say the solution flow $\varphi$ preserves the contact structure if it is a conformal contactomorphism:
\begin{equation}\label{preserve-contact}
\varphi_t^\ast\eta_\mathcal{L}=\lambda_t\eta_\mathcal{L},
\end{equation}
where $\lambda_t:\mathcal{L}\to\mathbb{R}$ is a nowhere zero function (called a conformal factor). We next prove this conclusion based on Lemma \ref{Lie derivative form} as well as the idea of random Hamiltonian, and we note that a coordinates proof for this case can be found in \cite{WeiWang2021}. Denote by $V_{h \cdot\delta X_t}$ the random vector field associated to $\varphi$ with the random contact Hamiltonian $h \cdot\delta X_t = \sum_{k=1}^r h_k \delta  X_t^k$. By Proposition \ref{Proposition-hvf}-($iv$), we have
\begin{equation}
  \mathscr{L}_{V_{h \cdot\delta X_t}}\eta_\mathcal{L} = -\mathcal{R}(h \cdot\delta X_t)\eta_\mathcal{L} = - \sum_{k=1}^r \mathcal{R}(h_k)\eta_\mathcal{L} \delta  X_t^k.
\end{equation}
Hence, the solution flow $\varphi$ satisfies
\begin{align*}
  d \varphi_t^\ast \eta_\mathcal{L}=\varphi_t^\ast \mathscr{L}_{V_{h \cdot\delta X_t}}\eta_\mathcal{L} = -\varphi_t^\ast  (\mathcal{R}(h \cdot\delta X_t)\eta_\mathcal{L})= - \left( \mathcal{R}(h_k) \circ \varphi_t \right) \varphi_t^\ast \eta_\mathcal{L} \delta  X_t^k,
\end{align*}
which can be regarded as a Stratonovich SDE on $\varphi_t^\ast \eta_\mathcal{L}$. By solving it, we conclude that \eqref{preserve-contact} holds with 
\begin{align} \label{conformal-factor}
\lambda_t=
\exp\left[{-\sum_{k=1}^r \int_0^t (\mathcal{R} (h_k) \circ \varphi_s)\delta  X_s^k}\right].
\end{align}

Again by Lemma \ref{foliation}, if $\mathrm{dim} (\mathcal{L})$ is even, then $\mathcal{L}$ is an L.C.S. manifold. That is, a manifold equipped with a nondegenerate, but not necessarily closed two-form $\Omega_\mathcal{L}\in \Omega^2(\mathcal{L})$; and, for each $z\in\mathcal{L}$, there is an open neighborhood $U \subset \mathcal L$ such that $d(e^\sigma \Omega_\mathcal{L})=0$, for some $\sigma\in C^\infty (U)$, so $(U, \omega_U\triangleq e^{\sigma} \Omega_\mathcal{L})$ is a symplectic manifold. For each $z\in\mathcal{L}$, we aim to show that $\varphi_t$ preserve the symplectic structure locally, that is,
\begin{equation}\label{preserve-cls}
\varphi_t^\ast\omega_U=\omega_U
\end{equation}
with fixed $U$ and $\sigma$. By \eqref{arbitrary-form}, the result \eqref{preserve-cls} follows immediately since $\mathscr{L}_{V_{h_k}}\omega_U=(\iota_{V_{h_k}}d+d\iota_{V_{h_k}})\omega_U=\iota_{V_{h_k}}d\omega_U+ddh_k=0$ for $k=1,\cdots r$. We remark that this result may be also proven globally by introducing the Lichnerowicz--de Rham differential. We also note that the problem reduces to a symplectic one if $U=\mathcal{L}$ and $d\sigma =0$, and refer to \cite{Milstein2002,WeiChaoDuan2019} for the coordinate proofs of the symplectic case.
\end{proof}

We note that, in the symplectic context, stochastic Hamiltonian flows preserve the symplectic form $\omega$ and hence the associated volume form $\omega^{n}=\omega \wedge \overset{n}{...} \wedge \omega$; see Bismut \cite{Bismut1981} and L\'azaro-Cam\'i $\&$ Ortega \cite{Lazaro2008}. Based on Theorem \ref{characteristic foliation}, we now formulate a contact version of the stochastic Liouville's theorem. For the deterministic contact Liouville's theorem, we refer to Bravetti $\&$ Tapias \cite{Bravetti2015}.

\begin{corollary}
\label{Liouville} (Stochastic contact/symplectic Liouville's Theorem) Let $\varphi$ be the stochastic symplectic Hamiltonian flow of the Hamiltonian semimartingale $(Z^h, \tau^h)$.
\begin{itemize}
  \item[(i)] Let $(M,\omega)$ be a symplectic manifold. Then, for all $t\in [0,\tau^h)$,
  \begin{align}
    \varphi_t^\ast  \omega^{n}=\omega^{n}.
  \end{align}
  \item[(ii)] Let $(M,\eta)$ be a contact manifold. Denote by $\Omega_\eta=\eta\wedge (d\eta)^{n}$ the contact volume element. Then, for all $t\in [0,\tau^h)$,
  \begin{align}
    \varphi_t^\ast \Omega_\eta =\lambda_t^{n+1} \Omega_\eta,
  \end{align}
  where $\lambda_t: M\to \mathbb{R}$ is the conformal factor of $\varphi_t$ given in \eqref{conformal-factor}. 
\end{itemize}
\end{corollary}

\begin{proof}
  The first assertion follows from the equality \eqref{arbitrary-form} in Lemma \ref{Lie derivative form} and the properties of symplectic forms, and has been proven in  \cite{Bismut1981,Lazaro2008}. We next only need to prove the second one. A straightforward computation using the product rule and Cartan's formula shows that
  \begin{align}
    \mathscr{L}_{V_{h \cdot\delta X_t}}\Omega_\eta &=- \mathcal{R}(h \cdot\delta X_t)\eta \wedge (d\eta)^{n}+n\eta \wedge (d\eta)^{n}  \wedge [-(d\mathcal{R}(h \cdot\delta X_t))\eta-\mathcal{R}(h \cdot\delta X_t)d\eta] \notag \\
    &= -(n+1)\mathcal{R}(h \cdot\delta X_t) \Omega_\eta,
  \end{align}
  where $V_{h \cdot\delta X_t}$ is the randomized Hamiltonian vector field with randomized Hamiltonian $h \cdot\delta X_t$. By applying the formula \eqref{arbitrary-form2}, and then solving the following equation
  \begin{align}
    \frac{d}{d t} \varphi_t^\ast \Omega_\eta=\varphi_t^\ast \mathscr{L}_{V_{h \cdot\delta X_t}}\Omega_\eta = -(n+1)\mathcal{R}(h \cdot\delta X_t)(\varphi_t) \varphi_t^\ast \Omega_\eta ,
  \end{align}
  we conclude the result as requested. 
\end{proof}

\subsection{Examples}

We now present some illustrative examples of stochastic Hamiltonian systems. For convenience, we always put the driven semimartingale $X$ as $X_t = (t, B_t)$ in the following examples, where $B$ is a standard Brownian motion, referring to the special case \eqref{SJHE-BM}.

\begin{example} (Stochastic Harmonic oscillators) The equation for Harmonic oscillator with noise
$\ddot{x}+kx=\sigma \dot{B}$ on $\mathbb{R}^1$, where the Hooke constant $k>0$ and the noise strength $\sigma>0$, is a standard stochastic symplectic Hamiltonian system if we introduce the velocity variable $y=\dot{x}$, the phase space $(M,\omega)=(\mathbb{R}^2,dx\wedge dy)$ and let $h_0(x,y)=\frac{1}{2}(kx^2+y^2)$, $h_1(x,y)=-\sigma x$. We note that the Stratonovich and It\^o differentials are identical as the coefficient $\sigma$ is a constant and the Wong--Zakai correction is equal to 0. Furthermore, for the initial state $(x_0, y_0)=(0, \sqrt{k})$, the system can be solved explicitly, i.e., 
 $x(t) = \cos(\sqrt{k} t) + \sigma \int_0^t \sin(\sqrt{k} (t-s)) dB_s$ and $y(t) = -\sqrt{k} \sin(\sqrt{k} t) + \sqrt{k} \sigma \int_0^t \cos(\sqrt{k} (t-s)) dB_s$. The second moment satisfies $\mathbb{E} \left[x(t)^2+ \frac{1}{k} y(t)^2\right] =1+\sigma^2t>0$. 
We can find that the random noise is indeed a non-conservative force, but, different from the usual ones, it may ``add" energy to the system instead of dissipating the system's energy. More generally, (finite-dimensional) symplectic Hamiltonian systems with (either additive or multiplicative) noise have been widely studied, especially on the long-time behaviors of the system, such as transience, scattering theory, averaging principle, and so on; see \cite{Freidlin2012,Wulm2001,Lixm2008,WeiChaoDuan2019,Guillin2023} and the references therein. 
\end{example}

\begin{example}\label{example2} (Stochatic dissipative/damping systems)
  Consider the damped mechanical system with noise $\ddot{x}+\gamma \dot{x}+ \nabla U(x)=\sigma \dot{B}_t$ on $\R^n$, where the damping coefficient $\gamma>0$, the noise strength $\sigma>0$ and the potential function $V$ is lower bounded. When $\sigma=0$ (i.e., no random perturbation), the system is dissipative and, by introducing $y=\dot{x}$, it will coverage to the phase points where the (symplectic) Hamiltonian $h_{\text{sym}}(x,y)=U(x)+\frac{1}{2}|y|^2$ attains the local minima \cite{Wulm2001}. When $\sigma>0$, the random noise will compensate for the loss of energy caused by the damping force, and such a stochastic dissipative system will approach some non-degenerate equilibrium measure. Referring to Roberts $\&$ Spanos \cite{Roberts2003}, the unique invariant measure for this particular system is $\mu(dx,dy)=\exp\big(-\frac{2\gamma}{\sigma^2}h_{\text{sym}}\big)dxdy$ which is equal to the Liouville measure $dxdy$ if $\gamma=0$. This system is usually called a Langevin one when $\sigma=\sqrt{2\gamma}$, and it covers many famous models, e.g., the generalized Duffing oscillator (where $U(x)$ is a lower-bounded polynomial).

  We now introduce an additional variable $u\in\mathbb{R}$ such that $(M,\eta)=(\mathbb{R}^{2n+1},du-\sum_{i=1}^n y_idx^i)$ forms a contact manifold. Define a contact Hamiltonian function $h_{\text{con}}(x,y,u)=h_{\text{sym}}(x,y)+\gamma u$. In this way, the system can be rewritten as
  \begin{equation}   
  \label{ex:model2}
  \begin{split}
  \left \{
  \begin{array}{ll}
    d{x}^i&=y_idt,\\
    d{y}_i&=-[\partial_{x_i} U(x)+\gamma y_i]dt+\sigma dB_t^i,\\
    d{u}&=\left[ \frac{1}{2}\sum_{i=1}^n y_i^2- U(x)- \gamma u \right] dt + \sum_{i=1}^n \sigma x^i dB_t^i,
  \end{array}
   \right.
    \end{split}
  \end{equation}
which a stochastic contact Hamiltonian system with $h_0=h_{\text{con}}$ and $h_k(x,y,u)=-\sigma x^k$, $k=1,2,\cdots,n$. The corresponding solution flow thus preserves the contact structure up to multiplication by a conformal factor $\lambda_t=\exp\left(-\mathcal{R}(h_0)t\right)=\exp(-\gamma t)$.
\end{example}

\begin{example}\label{example3} (Stochastic Gaussian isokinetic systems) The dynamics of the Gaussian isokinetic thermostat \cite{Dettmann1995,Ruelle2000} is governed by the following system on $\mathbb{R}^{2n}$:
  \begin{equation}   
  \label{ex:model3}
  \begin{split}
  \left \{
  \begin{array}{ll}
    \dot{q}^i&=p_i/m,\\
    \dot{p}_i&=f_i(q)+\zeta_i-\alpha(p) p_i,
  \end{array}
   \right.
    \end{split}
  \end{equation}
  for $i=1,2,\cdots,n$. Here $f$ is the field of an irrotational force and is usually assumed to be a gradient field $f=-\nabla U$ where $U:\R^n \to\mathbb{R}$ is a potential function. The term $\zeta$ stands for a perturbing force such that the system is out of equilibrium; $-\alpha(p) p$ is the thermostat ensuring that the unperturbed kinetic energy is constant, i.e., $0=\frac{d}{dt}\frac{|p|^2}{m}=\frac{p}{m}\cdot(f-\alpha(p) p)$ which gives $\alpha(p) = \frac{f\cdot p}{|p|^2}$. 
  As discussed by Wojtkowski $\&$ Liverani in \cite{Wojtkowski1998}, the unperturbed model can raise an L.C.S. structure when one restricts the systems to a level set of the Hamiltonian $h_0(q,p) = \frac{1}{2m} |p|^2$.
  
  We next show the perturbed model \eqref{ex:model3} is an L.C.S-type stochastic Hamiltonian system when we choose the random force $\zeta=\sigma(q) \diamond \dot{B}_t$ with $\sigma\in C^\infty(\R^n,\mathbb{R}^{n\times r})$. Consider a level set of the Hamiltonian $M_c\triangleq \{(q,p): h_0(q,p)=c \} \subset \mathbb{R}^{2n}$ for $c\in \mathbb{R}$. We can find that $M_c$ endowed with the 1-form $\theta_c=-\sum_{i=1}^n \frac{f_i}{2c} dq^i$ and the 2-form $\omega_c=\sum_{i=1}^n dq^i\wedge dp_i+\theta_c\wedge \sum_{j=1}^n p_jdq^j$ forms a L.C.S. manifold, since $d\omega_c=\theta_c\wedge\omega_c$. Actually, the manifold $(M_c,\omega_c,\theta_c)$ is globally conformal symplectic (G.C.S.) due to the exactness of $\theta_c=d(\frac{U}{2c})$. Denote by $V_{h_0}$ the vector field corresponding to the unperturbed version of \eqref{ex:model3}. We have $\omega_c(V_{h_0}, \cdot) = d h_0$ on $M_c$, that is, $V_{h_0}$ is the conformally Hamiltonian vector field associated with the Hamiltonian $h_0$ \cite{Wojtkowski1998}. Additionally, we introduce functions $h_k(q,p)\equiv h_k(q)$, $k=0,1,\cdots,r$, such that their conformally Hamiltonian vector fields are given by $V_{h_k} = \sum_{i=1}^n \sigma_k^i \frac{\partial}{\partial p_i}$, that is, $d h_k = \omega_c(V_{h_k}, \cdot) = - \sum_{i=1}^n \sigma_k^i dq^i$. This gives $h_k(q)= C_k -\int_0^q \sigma_k(x) dx$, where $C_k$'s are constants.
  In this way, the perturbed system \eqref{ex:model3} can be regarded as a stochastic Hamiltonian system with Hamiltonians $(h_k: k=0,1,\cdots,r)$ on the G.C.S. manifold $(M_c,\omega_c,\theta_c)$.
\end{example}

\section{The stochastic Hamilton--Jacobi--Bellman (HJB) theory}

We note that the Hamilton--Jacobi theory is an alternative formulation of classical mechanics, equivalent to the Hamiltonian one. In this section, we aim to show that an analogous result holds for the stochastic setting. More precisely, based on Theorem \ref{characteristic foliation}, we will discuss contact and L.C.S. cases, respectively. 

First of all, we briefly review the core idea of deterministic Hamilton--Jacobi theory on the symplectic manifold. Let $Q$ be an $n$-dimensional smooth manifold serving as the configuration space. Denote by $TQ$ and $T^\ast Q$ the tangent and cotangent bundles over $Q$ respectively. Consider the canonical Liouville 1-form $\theta_Q$ ($=p_idq^i$ locally) on $T^\ast Q$. The canonical symplectic 2-form on $T^\ast Q$ is given by $\omega_Q=-d\theta_Q$ ($=dq^i\wedge dp_i$ locally). We remark that the Einstein summation convention would be used (especially for the expressions of local coordinates) when there is no ambiguity. The classical Hamilton--Jacobi theory consists of finding a principal (or generating) function $S(q,t)$, that fulfills
\begin{equation}\label{H-J}
    \frac{\partial S}{\partial t}+h_0\left( q,\frac{\partial S}{\partial q},t \right)=0,\quad i=1,\cdots,n,
\end{equation}
where $h_0\in C^\infty (T^* Q\times\R)$ is the Hamiltonian function of the system. For the important situation that $h_0$ is independent of time and $S$ is separable in time (i.e., $S(q,t) = W(q) - t \cdot \text{const}$), one easily see that \eqref{H-J} can be rewritten as
\begin{equation}\label{H-J-2}
  d(h_0 \circ dW) = 0,
\end{equation}
where the 1-form $dW$ on $Q$ is understood as a section of the cotangent bundle. This 1-form transforms the integral curves of a vector field on $Q$ into integral curves of the Hamiltonian vector field $V_{h_0}$ on $T^\ast Q$. Such a geometric procedure has been extended to contact systems, L.C.S. systems, and many other contexts, as mentioned in the introduction.

\subsection{Stochastic HJ equations for contact cases}

In this subsection, we assume that the underlying contact manifold $(M,\eta)$ is the extended phase space endowed with its canonical contact structure. Namely, $M=T^{\ast}Q\times \mathbb{R}$ with $\mathrm{dim}(M)=2n+1$, and $\eta=du-\rho^\ast \theta_Q$ with $\rho:T^{\ast}Q\times \mathbb{R}\to T^{\ast}Q$ being the canonical projection. We will show that there are two approaches to developing the HJ theory for stochastic contact Hamiltonian systems, via considering sections of $\pi_Q:M \to Q$ and $\pi_{\widehat{Q}}:M \to Q\times\mathbb{R}$ respectively.

\subsubsection{A contact HJB formalism on $Q$}

Let $\pi_Q:J^1Q\to Q$ be the space of 1-jets $j^1 f$ of smooth functions (or germs) $f:Q\to \mathbb{R}$ \cite{Geiges2008,Lee2013}. One can identify $J^1Q$ with the space $T^{\ast}Q\times \mathbb{R}$ by mapping $j^1f(q)$ to $(df_q,f(q))$. A section $\Gamma: Q\to J^1Q$ is called Legendrian if its image is a Legendrian submanifold of $J^1Q$, or equivalently, $\Gamma^\ast \eta=0$. The following lemma is taken from \cite[Proposition 22.37]{Lee2013} or \cite[Proposition 3]{deLeon2021contact}.

\begin{lemma}\label{1-jet}
 A local section of $J^1Q$ is Legendrian if and only if it is the 1-jet of a function.
\end{lemma}

Legendrian submanifolds play an important role in contact geometry, similar to Lagrangian submanifolds in symplectic geometry. In fact,  the image of a contact Hamiltonian vector field, suitably included in the contactified tangent bundle, is a Legendrian submanifold \cite{deLeon2019JMP}. Based on Theorem \ref{characteristic foliation} or Corollary \ref{Liouville}, the stochastic Hamiltonian flow $\varphi_t$ is indeed a contactomorphism, that is, there exist a nowhere zero function $\lambda_t$ given in \eqref{conformal-factor} such that $\varphi_t^\ast \eta=\lambda_t \eta$. We thus have the following lemma.

\begin{lemma}\label{Legend}
    Let $L\subset J^1Q$ be Legendrian. Then $\varphi_t(L)\subset J^1Q$ is also Legendrian.
\end{lemma}
 
To compare with the deterministic scenario and make things clearer, we start from the random Hamiltonian vector field $V_{h \cdot\delta X_t}$ introduced in 
\eqref{SJHE3}, and focus on our arguments with fixed time $t\in\mathbb{R}_+$ and sample $\omega\in{\mathbf\Omega}$. In this way, the randomized Hamiltonian $h \cdot\delta X_t$ is regard as a function from $J^1Q$ to $\mathbb{R}$; See the diagram below:
$$
    \centering
    \begin{tikzcd}[column sep=scriptsize, row sep=scriptsize]
  J^1Q\cong T^{\ast}Q\times \mathbb{R}
  \arrow[ddrr, bend left=35, "h \cdot\delta X_t"]
  \arrow[ddrr, swap, "\pi_u"]
  \arrow[dd, swap, "\rho"] & & \\
  && &\\
   T^{\ast}Q  & 
      & \mathbb{R}.
\end{tikzcd}
$$
We consider a time-dependent section $\Gamma_t:Q\to J^1Q$ such that $\Gamma_t=j^1S_t=(dS_t,S_t)$ for a function $S:Q\times\mathbb{R}_+\to \mathbb{R}$. By Lemma \ref{1-jet}, $\Gamma_t(Q)$ is a Legendrian submanifold of $(J^1Q,\eta)$, which implies that $(\rho \circ \Gamma_t)(Q)$ is a Lagrangian submanifold of $(T^{\ast}Q,\omega_Q)$. Furthermore, given the randomized Hamiltonian vector field $V_{h \cdot\delta X_t}$ on $J^1Q$, we can define a projected vector field $V_{h \cdot\delta X_t}^{\Gamma}$ on $Q$ as follows:
\begin{equation}\label{randomized vector field}
    V_{h \cdot\delta X_t}^\Gamma =T\pi_Q \circ V_{h \cdot\delta X_t} \circ \Gamma_t,
\end{equation}
where $T\pi_Q$ is the tangent mapping from $T(J^1Q)$ to $TQ$. We summarize the discussion for this definition of $V_{h \cdot\delta X_t}^\Gamma$ in the following commutative diagram:
$$
     \begin{tikzcd}[column sep=scriptsize, row sep=scriptsize]
  J^1Q
  \arrow[rr, "V_{h \cdot\delta X_t}"]
  \arrow[dd, "\pi_Q"] 
  & & T(J^1Q)
  \arrow[dd, "T\pi_Q"] \\
  &&\\
 Q
   \arrow[rr, "V_{h \cdot\delta X_t}^{\Gamma}"]
   \arrow[uu, bend left=50, "\Gamma_t"]
   &   & TQ.
\end{tikzcd}
$$
Base on the idea of the classical Hamilton--Jacobi theory, on the one hand, we want $\Gamma_t$ to fulfill 
\begin{equation}\label{Gamma-related}
    V_{h \cdot\delta X_t} \circ \Gamma_t = T\Gamma_t \circ V_{h \cdot\delta X_t}^\Gamma,
\end{equation}
i.e., the random vector fields $V_{h \cdot\delta X_t}$ and $V_{h \cdot\delta X_t}^\Gamma$ are $\Gamma$-related. This means that, provided with an integral process of $V_{h \cdot\delta X_t}^\Gamma$, one can obtain an integral process of the random Hamiltonian vector field $V_{h \cdot\delta X_t}$, hence, a solution of the stochastic contact Hamiltonian system. We refer to such a solution as a horizontal one since it is indeed on the image of a 1-jet on $Q$. On the other hand, we attempt to derive an equivalent equation, which is similar to \eqref{H-J-2}, from the condition \eqref{Gamma-related}. That is, we will obtain an equation for the principal function $S$ (also referred to as the contact generating function).

Now suppose that $L\subset J^1Q$ is the graph of $j^1S_0$, so $S_0$ is a contact generating function for $L$. By Lemmas \ref{1-jet} and \ref{Legend}, it is expected that, at least for a short time, $\varphi_t(L)$ is the graph of $j^1S_t$ for $S_t$ a function on $Q$. We can adjust
the arbitrary constants in $S_t$ so it depends smoothly on $t$ and coincides with $S_0$ at $t=0$. To this end, we require the following preparation.

\begin{lemma}\label{projected-lem}
    Let $L\subset J^1Q$ be Legendrian. Let $z_0\in L$ be a regular point (that is, the $T_{z_0}\pi_Q|_L:T_{z_0}L\to T_{\pi_Q({z_0})}Q$ is an isomorphism) and $q_0=\pi_Q({z_0})\in Q$ be the projective image. Let $\varphi$ be the stochastic Hamiltonian flow of \eqref{SJHE} with the initial point $z_0$, and let $\tau^h$ be the maximal stopping time. Then, there exist two neighborhoods $\mathcal{U}_{z_0}\subset L$, $\mathcal{U}_{q_0}\subset Q$, of $z_0$ and $q_0$ respectively, and a random map $\tau:\mathcal{U}_{z_0}\times {\mathbf\Omega}\to \mathbb{R}_+$ such that, for any $z\in \mathcal{U}_{z_0}$, $\tau_{z}=\tau(z):{\mathbf\Omega}\to \mathbb{R}_+$ is a stopping time and the equation
    \begin{align}\label{projected-eq}
        \pi_Q(\varphi_t(z,\omega))=q,\quad t\in [0,\tau_{z}),
    \end{align}
    has a unique solution in $\mathcal{U}_{q_0}$, denoted by $q=\psi_t(z,\omega)$.
    $$
     \begin{tikzcd}[column sep=scriptsize, row sep=scriptsize]
     z\in\mathcal{U}_{z_0}\subset L\subset J^1Q
     \arrow[rr, "\varphi_t"]
     \arrow[ddrr, "\psi_t"] & & 
     \varphi_t(z,\omega)\in L\subset J^1Q
     \arrow[dd, "\pi_Q"] \\
     & &\\
     & & q=\psi_t(z,\omega)\in \mathcal{U}_{q_0}\subset Q.
     \end{tikzcd}
     $$ 
    Furthermore, $\psi_\cdot(z):[0,\tau_{z})\times {\mathbf\Omega}$ is a semimartingale for any $z\in \mathcal{U}_{z_0}$, and $\psi_t(\cdot,\omega): \mathcal{U}_{z_0}\to \mathcal{U}_{q_0}$ is a diffeomorphism for any $t\in [0,\tau_{z})$ which depends continuously on $t$.
\end{lemma}
\begin{proof}
    The key idea is to construct the solution by leveraging the regularity of the initial point $z_0$ and the continuity of the differential of the stochastic Hamiltonian flow $\varphi$. The properties of the solution are then established by applying the Implicit Function Theorem and the Stratonovich differential rules. Notably, similar arguments in the symplectic setting have been employed by L\'azaro-Cam\'i $\&$ Ortega \cite{LazaroCam2009}. With slight modifications, analogous results can also be obtained by replacing $Q$ with $Q\times \mathbb{R}$ or $J^1Q$ with $T_\theta^\ast Q$.

    We begin by selecting an open and sufficiently small neighborhood $U_{z_0}\subset L$ of $z_0$, such that $\pi_Q|_{U_{z_0}}$ is a diffeomorphism onto its image. Furthermore, we choose a set of local coordinates $\{q^1,q^2,\cdots,q^n\}$ on $U_{q_0}=\pi_Q(U_{z_0})$. By Lemma \ref{1-jet}, the induced coordinates on $U_{z_0}$ take the form $\{q^i,\frac{\partial S_0}{\partial q^i}, S_0\}$ for some function $S_0\in C^\infty(U_{q_0};\mathbb{R})$.
    
    For any $z\in U_{z_0}$, we define the first exit time at which the process  $\pi_Q(\varphi(z))$ exits from $U_{q_0}$ as 
    $$
    \tau_{U_{q_0}} (z,\omega)=\inf \{0\leq t <\tau^h\;|\;\pi_Q\circ \varphi_t(z,\omega)\notin U_{q_0}\}.
    $$
    Consider the map $\Psi_t=\pi_Q\circ \varphi_t$ for $0\leq t< \tau_{U_{x_0}}(z,\omega)$. Since $z_0\in L$ is a regular point, it follows that $\det\left(\frac{\partial \Psi_0^i}{\partial q^i}(z_0)\right)\neq 0$ a.s., for $i=1,\cdots,n$. Moreover, the continuity of the derivative of $\Psi_0(\cdot,\omega):U_{z_0}\to U_{q_0}$ ensures the existence of a neighborhood $\mathcal{U}_{z_0}\subset U_{z_0}$ such that 
    $$
    \det\left(\frac{\partial \Psi_0^i}{\partial q^i}(z)\right)> 0\;\;\text{a.s., for any } z\in \mathcal{U}_{z_0}.
    $$
    Next, we define the map  $\Xi_t(z,\omega)=\det\left(\frac{\partial \Psi_t^i}{\partial q^i}(z,\omega)\right):\mathcal{U}_{z_0}\times [0,\tau_{U_{q_0}} (z,\omega))\times {\mathbf\Omega} \to \mathbb{R}.$
    It is evident that for any $z\in \mathcal{U}_{z_0}$, $\Xi(z)$ is a well-defined and continuous semimartingale with $\Xi_0(z)>0$, due to the continuity of the differential of the stochastic Hamiltonian flow $\varphi$. This allows us to define a first exit time as
    $$
    \tau(z,\omega)=\inf \{0<t<\tau_{U_{q_0}}\;|\;\Xi_t(z,\omega)\notin \mathbb{R}_+\},
    $$
    To establish the existence and uniqueness of solutions to equation \eqref{projected-eq}, it suffices to solve the equation
    \begin{align}\label{projected-eq-2}
        \pi_Q\left(\varphi_{t\wedge \tau}(z,\omega) \right)=q,
    \end{align}
    where ``$\wedge$" denotes the minimum operation. Noting that $\varphi_{t\wedge \tau}(z)\in U_{z_0}$ if $z\in \mathcal{U}_{z_0}$, we can describe $\varphi_{t\wedge \tau}(z)$ in terms of the chosen local coordinates. Defining $\tau_z=\tau(z,\omega)$ and $\mathcal{U}_{x_0}=\pi_Q(\mathcal{U}_{z_0})$, an application of the Implicit Function Theorem guarantees the existence of a unique solution $\psi_t(z,\omega)$ to equation \eqref{projected-eq-2}.
    
    Furthermore, equation \eqref{projected-eq-2} can be rewritten in its Stratonovich differential form as
    \begin{align}\label{projected-eq-2-stratonovich}
        \delta\psi_t(z)=\sum_{k=1}^d\left[T_{\psi_t(z)}\Psi\right]^{-1}\left(
        T_{\varphi_{t\wedge \tau}\left(\psi_t(z)\right)}\left(q \circ \pi_Q\right)
        \Big(V_{h_k}\big(\varphi_{t\wedge \tau}\left(\psi_t(z)\right)\big) \right)
        \Big) (dt, \delta  \xi_t^k).
    \end{align}
    Here $q$ should be understood as the local chart maps associated with the chosen coordinate system. The properties of $\psi$ follow naturally, as $\psi_t(z,\omega)$ is indeed the unique stochastic flow associated with the SDE \eqref{projected-eq-2-stratonovich}.
\end{proof}

\begin{remark}
    The preceding lemma ensures that we can always work in a neighborhood of regular points of the Legendrian submanifold $L$ and project the Hamiltonian semimartingale, parameterized by $J^1Q$, onto a process defined on the base manifold $Q$, up to a suitable stopping time. However, for sufficiently large $t$, the image $\varphi_t(L)$ may cease to be the graph of the differential of a function; in particular, the function $S(q,t)$ may blow up in finite time. In this case, one says that a caustic forms \cite{Marsden1978}. Nevertheless, the interpretation of $\varphi_t(L)$ as a Legendrian submanifold remains valid; it may simply ``bend over", failing to maintain the structure of a graph. For a more in-depth discussion on contact geometry, we refer to \cite{Geiges2008} and the references therein.
\end{remark}

For a 1-form $\alpha$ on $M=T^{\ast}Q\times \mathbb{R}$ with local expression $\alpha = \alpha_i d q^i + \alpha^i d p_i + \alpha_0 dz$, we define its \textit{horizontal projection} as the tangent vector field $\alpha_{\mathbf h}$ on $Q$, given locally by
\begin{equation}\label{local-eprs}
  \alpha_{\mathbf h} = \alpha^i \pt_{q^i}.
\end{equation}
Similarly, the \textit{vertical lift} of $\alpha$, denoted by $\alpha^{\mathbf{v}}$, is given by
\begin{equation}
    \alpha^{\mathbf{v}}=-\alpha_i\frac{\partial}{\partial p_i}.
\end{equation}
We have the following lemma, which is a variant of \cite[Lemma 5.2.5]{Marsden1978} in the symplectic setting.

\begin{lemma}\label{lemma-1}
Let $S$ be a smooth function on $Q$. Then
\begin{itemize}
  \item[(i)] $(j^1S \circ \pi_Q)_* \mathcal R = 0$.
  \item[(ii)] for any smooth function $f: M=T^{\ast}Q\times \mathbb{R}\to \R$ and 1-form $\alpha$ on $M$,
\begin{equation*}
  \Lambda\left( (j^1 S\circ \pi_Q)^* \alpha - \alpha|_{j^1 S}, df \right) = d(f\circ j^1 S) (\alpha|_{j^1 S})_{\mathbf{h}}.
\end{equation*}
\end{itemize}
\end{lemma}

\begin{proof}
We first note that the map $j^1 S\circ \pi_Q: M \to M$ has the following local expression:
\begin{equation*}
  j^1 S\circ \pi_Q(q,p,z) = (q, dS(q), S(q)).
\end{equation*}
Assertion (i) follows from the fact that $\mathcal R$ has the local expression $\mathcal R = \pt_u$. To prove (ii), we assume that $\alpha$ has the local expression \eqref{local-eprs}. Then
\begin{equation*}
  (j^1 S\circ \pi_Q)^* \alpha = \alpha_i|_{j^1 S} d q^i + \alpha^i|_{j^1 S} d \pt_i S + \alpha_0|_{j^1 S} dS,
\end{equation*}
and thus,
\begin{equation*}
  (j^1 S\circ \pi_Q)^* \alpha - \alpha|_{j^1 S} = \left( \alpha^i|_{j^1 S} \pt_i \pt_j S + \alpha_0|_{j^1 S} \pt_j S \right) dq^j - \alpha^i|_{j^1 S} dp_i - \alpha_0|_{j^1 S} dz.
\end{equation*}
On the other hand, as
\begin{equation*}
  \Lambda = \left( \pt_{q^i} + p_i \pt_z \right) \wedge \pt_{p_i},
\end{equation*}
we have
\begin{equation*}
  \begin{split}
    \Lambda\left( (j^1 S\circ \pi_Q)^* \alpha - \alpha|_{j^1 S}, df \right) &= \left( \alpha^i|_{j^1 S} \pt_i \pt_j S \right) \pt_{p_i} f|_{j^1 S} + \alpha^i|_{j^1 S} \left( \pt_{q^i} + \pt_i S \pt_z \right) f|_{j^1 S} \\
    &= \alpha^i|_{j^1 S} \pt_{q^i} [f(q, dS, S)].
  \end{split}
\end{equation*}
This proves the result.
\end{proof}

We next present one of the main results of this section.

\begin{theorem} (Stochastic contact HJ formalism) \label{SCJH-1}
  Let $M=J^1Q$ be endowed with the canonical contact structure $\eta$. Consider the stochastic Hamiltonian system given in \eqref{SJHE}.  Let $S:Q\times [0,\tau_Q) \to \mathbb{R}$ with $S|_{t=0} = f\in C^\infty(Q)$ and $\tau_Q$ the maximal stopping time of \eqref{Q-curve}.  Then, for each $t\in [0,\tau_Q)$, the following statements are equivalent:
  \begin{itemize}
    \item[($i$)] for every semimartingale $q(t)$ in $Q$ satisfying
    \begin{align}\label{Q-curve}
    \delta q(t) = (\pi_Q)_\ast  \mathcal{H}\big(X_t,j^1S(q(t),t)\big) \delta  X_t,
    \end{align}
    the process $(\omega,t)\mapsto j^1S(q(t),t)$ is a Hamiltonian semimartingale solving \eqref{SJHE}. 
    \item[($ii$)] $S$ satisfies the stochastic contact HJB equation
    \begin{equation}\label{SHJE-1-0}
    \left\{
     \begin{aligned}
        &\big[d(\pt_t S) \big]^{\mathbf{v}} dt + \big[d {h}\left(j^1S\right)\big]^{\mathbf{v}}\cdot\delta  X_t=0,\\
        & (\pt_t S)\mathcal{R} dt +  h\left(j^1S\right)\mathcal{R}\cdot\delta  X_t=0,
    \end{aligned}
    \right.
    \end{equation}
     where $[\cdot]^{\mathbf{v}}$ is the vertical lift operator. Equivalently, equation \eqref{SHJE-1-0} can be rewritten as
    \begin{align}\label{SHJE-1}
     \pt_t S dt + {h} \left(j^1S\right) \cdot \delta  X_t=0.
    \end{align}
\end{itemize} 
\end{theorem}

\begin{proof}
Denote $\Gamma_t := j^1S(q(t),t)=(dS(q(t),t),S(q(t),t))$, where $q(t)$ satisfies the stated equation \eqref{Q-curve}. Then, by the chain rule, for $\forall \alpha\in \Omega(J^1Q)$,
\begin{equation}\label{d-t-gamma}
\begin{split}
  \int \alpha (\delta \Gamma(t))
  =&\ \int \alpha \left( (j^1S)_\ast \delta q(t) + \pt_t (j^1S) dt \right) \\
  =&\ \int \alpha \left( (j^1S \circ \pi_Q)_\ast \mathcal{H}\big(X_t, \Gamma_t\big) \delta X_t \right) \\
  & + \int \rho_\ast \left( \alpha|_{\Gamma_t} \right) \left( \pt_t dS(q(t),t) \right) dt + \int (\pi_u)_\ast \left( \alpha|_{\Gamma_t} \right) \left( \pt_t S(q(t),t) \right) dt \\
  =&\ \underbrace{\int \mathcal{H}^\ast\big(X_t, \Gamma_t\big) (j^1S\circ \pi_Q)^\ast \alpha (\delta X_t)}_{=:I} \\
  & + \int  d (\pt_t S) \left( \alpha|_{\Gamma_t} \right)_\mathbf{h} (q(t),t) dt + \int \pt_t S (q(t),t) \alpha (\mathcal{R}) (\Gamma_t) dt.
\end{split}
\end{equation}
What we need to show is that \eqref{d-t-gamma} is equal to $\int \mathcal{H}^\ast\big(X_t,\Gamma_t\big)(\alpha) (\delta  X_t)$ if and only if statement ($ii$) holds. Now, it follows from \eqref{dual-Stratonovich} that
\begin{equation}\label{dual-Stratonovich-1}
  \int \mathcal{H}^\ast\big(X_t,\Gamma_t\big)(\alpha) (\delta X_t) = \int \Lambda(dh_k, \alpha) (\Gamma_t) \delta X_t^k - \int h_k(\Gamma_t) \alpha(\mathcal R) (\Gamma_t) \delta X_t^k,
\end{equation}
and
\begin{equation*}
  I = \int \Lambda\left( dh_k, (j^1S\circ \pi_Q)^\ast \alpha \right) (\Gamma_t) \delta X_t^k - \int h_k(\Gamma_t) (j^1S\circ \pi_Q)^\ast \alpha(\mathcal R) (\Gamma_t) \delta X_t^k.
\end{equation*}
By Lemma \ref{lemma-1}-(ii),
\begin{equation*}
  \int \Lambda\left( dh_k, (j^1S\circ \pi_Q)^\ast \alpha -\alpha \right) (\Gamma_t) \delta X_t^k = \int d(h_k(\Gamma_t)) (\alpha|_{\Gamma_t})_{\mathbf{h}} \delta X_t^k.
\end{equation*}
By Lemma \ref{lemma-1}-(i),
\begin{equation*}
  (j^1S\circ \pi_Q)^\ast \alpha(\mathcal R) = \alpha\left( (j^1S\circ \pi_Q)_\ast \mathcal R \right) = 0.
\end{equation*}
Thus,
\begin{equation*}
  I = \int \Lambda(dh_k, \alpha) (\Gamma_t) \delta X_t^k + \int d(h_k(\Gamma_t)) (\alpha|_{\Gamma_t})_{\mathbf{h}} \delta X_t^k.
\end{equation*}
Plugging this and \eqref{dual-Stratonovich-1} into \eqref{d-t-gamma}, we get
\begin{equation*}
  \begin{split}
    \int \alpha (\delta \Gamma(t)) =&\ \int \mathcal{H}^\ast\big(X_t,\Gamma_t\big)(\alpha) (\delta X_t) + \int \alpha(\mathcal R) (\Gamma_t) \left[ h_k(\Gamma_t) \delta X_t^k + \pt_t S (q(t),t) dt \right] \\
    &\ + \int d(h_k(\Gamma_t)) (\alpha|_{\Gamma_t})_{\mathbf{h}} \delta X_t^k + \int  d (\pt_t S) (q(t),t) \left( \alpha|_{\Gamma_t} \right)_\mathbf{h} dt.
  \end{split}
\end{equation*}
The arbitrariness of $\alpha$ and the independence of $\alpha(\mathcal R)$ and $(\alpha|_{\Gamma_t})_{\mathbf{h}}$ imply that $\int \alpha (\delta \Gamma(t)) = \int \mathcal{H}^\ast\big(X_t,\Gamma_t\big)(\alpha) (\delta X_t)$ if and only if
\begin{equation*}
  h_k(\Gamma_t) \delta X_t^k + \pt_t S (q(t),t) dt = 0,
\end{equation*}
and
\begin{equation*}
  d(h_k(\Gamma_t)) \delta X_t^k + d (\pt_t S) (q(t),t) dt = 0.
\end{equation*}
This yields the result.
\end{proof}

In particular, for the situation that $X_t = (t, B_t)$ with $B$ being a standard Brownian motion, this stochastic contact HJ equation has the following local expression:
\begin{align}\label{SHJE-1-local}
  S(q,t) =& S_0 (q)-\int_0^t h_0\bigg(q, \frac{\partial S_s(q)}{\partial q},S_s(q)\bigg)ds -\sum_{k=1}^r\int_0^t h_k\bigg(q, \frac{\partial S_s(q)}{\partial q},S_s(q)\bigg) \delta B_s^k.
\end{align}

\begin{remark}
The motivation of the formulation for Theorem \ref{SCJH-1} as well as the proof above comes from Abraham $\&$ Marsden \cite[Theorems 5.2.4 and 5.2.18]{Marsden1978}. Notice that the proof also works for the case that $h=(h_0,h_1,\cdots,h_r):M\times \mathbb{R}_+\to \mathbb{R}\times \R^r$ dependent on $t$ explicitly, and still works in infinite dimensions which means Theorem \ref{SCJH-1} may be applied to quantization problems for field theories.
 \end{remark}
  \begin{remark} 
  Referring to de Le\'on $\&$ Sard\'on \cite{deLeon2017} and de Le\'on, Lainz  $\&$  Mu\~niz-Brea \cite{deLeon2021contact}, we also give a coordinate-based proof of Theorem \ref{SCJH-1}. In Darboux coordinates $Z=(q^i,p_i,u)$, by Lemma \ref{1-jet}, we can write the section $\Gamma:Q\times \mathbb{R}_+\to J^1Q\times \mathbb{R}_+$ (which can be understood as a curve) as $\Gamma=j^1S$, that is,
     \begin{align}\label{local-Gamma}
    \Gamma(q^i,t)=\big(q^i, \Gamma_j(q,t), \Gamma_u(q,t),t\big)
    = \left(q^i, \frac{\partial S(q,t)}{\partial q^j}, S(q,t),t\right),
    \end{align}
    where $\Gamma_j := p_j \circ \Gamma$ and $\Gamma_u := u \circ \Gamma$.
    Note that the local expression of the random Hamiltonian vector field is
    \begin{align}\label{local-X-H}
    V_{h \cdot\delta X} = \mathcal{H}(X,Z) (\delta X)
    =&\ \sum_{k=1}^d \frac{\partial h_k}{\partial p_i} \delta X^k \frac{\partial}{\partial q^i} \notag\\
    &\ - \sum_{k=1}^d\left(\frac{\partial h_k}{\partial q^i}+p_i\frac{\partial h_k}{\partial u}\right) \delta X^k \frac{\partial}{\partial p_i}\notag\\
    &\ + \sum_{k=1}^d  \left(p_i\frac{\partial h_k}{\partial p_i}-h_k\right) \delta X^k \frac{\partial}{\partial u},
    \end{align}
    and the vector field $V_{h \cdot\delta X_t}^{\Gamma}:Q\to TQ$ can be defined by
    \begin{align}\label{local-X-H-Gamma}
    V_{h \cdot\delta X_t}^{\Gamma}=T\pi_Q \circ V_{h \cdot\delta X_t} \circ {\Gamma_t}= \sum_{k=1}^d\bigg( \frac{\partial h_k}{\partial p_i}\circ {\Gamma_t} \bigg) \delta X^k_t \frac{\partial}{\partial q^i}.
    \end{align}
    Taking the time variable into account, we may introduce the corresponding time-dependent vector fields as 
    $$
    \widetilde{V}_{h \cdot\delta X_t}=V_{h \cdot\delta X_t}+\frac{\partial}{\partial t}\quad \;\text{and}\quad \;\widetilde{V}_{h \cdot\delta X_t}^{\Gamma}=V_{h \cdot\delta X_t}^{\Gamma}+\frac{\partial}{\partial t}.
    $$
    Then, the two statements in Theorem \ref{SCJH-1} become
    \begin{itemize}
      \item[($i$)] The two vector fields $\widetilde{V}_{h \cdot\delta X_t}$ and $\widetilde{V}_{h \cdot\delta X_t}^{\Gamma}$ are $\Gamma$-related, i.e., 
        \begin{align}\label{Gamma-related-t}
        \widetilde{V}_{h \cdot\delta X_t} \circ \Gamma = T\Gamma \circ \widetilde{V}_{h \cdot\delta X_t}^\Gamma;
        \end{align}
        \item[($ii$)] The following equation is fulfilled
    \begin{align}\label{SHJE-1-0-local}
     \left[\partial_t (\rho\circ \Gamma) \right]^\mathbf{v} dt + \left[ d(h \circ \Gamma_t) \right]^\mathbf{v} \cdot \delta  X_t= 0.
    \end{align}
    \end{itemize}
Here $\partial_t (\rho\circ \Gamma)$ is the tangent vector at a point $q$  associated with the curve; see the following commutative diagram:
    $$
    \centering
    \begin{tikzcd}[column sep=scriptsize, row sep=scriptsize]
  J^1Q\times \mathbb{R}_+
  \arrow[ddrr, "\pi_+"]
  \arrow[dd, "{\pi_Q\times id}"] & & \\
  && &\\
   Q \times \mathbb{R}_+ 
   \arrow[uu, bend left=35, "\Gamma"]
   \arrow[rr, swap, ""]
   & & 
   \mathbb{R}_+.
   \arrow[uull, swap, bend left=-30, "(\rho\circ\Gamma)"]
\end{tikzcd}
$$   

We now verify the equivalence of (i) and (ii) by local coordinates. On the one hand, it is a matter of calculation to check that, by
$$
T{\Gamma}\bigg( \frac{\partial}{\partial t}\bigg)
   =\frac{\partial {\Gamma}_j}{\partial t} \frac{\partial}{\partial p^j}+ \frac{\partial \Gamma_u}{\partial t} \frac{\partial}{\partial u}+\frac{\partial}{\partial t}
$$
and
$$
T{\Gamma}\bigg( \frac{\partial}{\partial q^i}\bigg)
   =\frac{\partial}{\partial q^i}+ \frac{\partial {\Gamma}_j}{\partial q^i} \frac{\partial}{\partial p^j}+ \frac{\partial \Gamma_u}{\partial q^i} \frac{\partial}{\partial u},
$$
equation \eqref{Gamma-related-t} holds if and only if
  \begin{equation}   
  \label{eq:hj} 
\left\{ 
\begin{aligned} 
    -\left[\left(\frac{\partial h}{\partial q^j}+\Gamma_j \frac{\partial h}{\partial u}\right)\cdot \delta X_t \right]\frac{\partial}{\partial p_j}
    &= \left[\frac{\partial \Gamma_j}{\partial t}dt +\left( \frac{\partial h}{\partial p_i}\cdot \delta X_t \right)\frac{\partial \Gamma_j}{\partial q^i}\right]\frac{\partial}{\partial p_j},
\cr 
\left[\left(\Gamma_j\frac{\partial h}{\partial p_j}-h\right)\cdot \delta X_t \right]\frac{\partial}{\partial u}
    &=\left[\frac{\partial \Gamma_u}{\partial t}dt +\left(\frac{\partial h}{\partial p_j}\cdot \delta X_t \right)  \frac{\partial \Gamma_u} {\partial q^j} \right]\frac{\partial}{\partial u}.
\end{aligned} 
\right. 
\end{equation}
On the other hand, in coordinates, one has
$$
[\pt_t (\rho\circ \Gamma)]^\mathbf v =\left(\frac{\partial \Gamma_j}{\partial t}dq^j\right)^\mathbf v =-\frac{\partial \Gamma_j}{\partial t}\frac{\partial}{\partial p_j}
$$
and
$$
[d(h \circ \Gamma_t)]^\mathbf{v} \cdot\delta X_t = -\left[\frac{\partial h_k}{\partial q^i} + \frac{\partial h_k}{\partial p_j}\frac{\partial \Gamma_j}{\partial q^i}+\frac{\partial h_k}{\partial u}\frac{\partial \Gamma_u}{\partial q^i}\right] \delta X_t^k \frac{\partial}{\partial p_i}.
$$
 We thus can find that \eqref{Gamma-related-t} or \eqref{eq:hj} is equivalent to \eqref{SHJE-1-0-local} if and only if $\Gamma=j^1S=(dS,S)$ is the 1-jet of the function $S$, that is, $\Gamma$ satisfies \eqref{local-Gamma} with
     $$
    \Gamma_j=\frac{\partial \Gamma_u} {\partial q^j}\quad \;\text{and}\quad \;\Gamma_u=S.
    $$
Equation \eqref{SHJE-1-0-local} is referred to as a stochastic Hamilton--Jacobi equation with respect to sections. We can also rewrite it as
\begin{align}\label{SHJE-1-local-0}
    \pt_t S dt + h \circ (dS_t,S_t) \cdot \delta  X_t = 0.
\end{align}
\end{remark}

\subsubsection{An alternative formalism on $Q\times \mathbb{R}$}

    We notice that equation \eqref{SHJE-1} or \eqref{SHJE-1-local} are similar to the symplectic case; see \cite[Theorem 3.4 or Example 1]{LazaroCam2009}. One reason behind this is that the HJB formalisms above are considered on the configuration space $Q$. However, we may have to consider the problem on the space $Q\times \mathbb{R}$ when the $(2n+1)^{st}$-dimensional variable can not be expressed by the configuration variable \cite{deLeon2017,deLeon2021contact}. Next, we introduce an alternative approach to ``reduce" the Hamiltonian formulation to an HJB one. That is, instead of considering sections of $\pi_Q: T^\ast Q\times \mathbb{R}\to Q$, we next focus on a section of the projection $\pi_{\widehat{Q}}: T^\ast Q\times \mathbb{R}\to Q\times \mathbb{R}$, say $\widehat{\Gamma}_t:Q\times \mathbb{R}\to T^\ast Q\times \mathbb{R}$. See the commutative diagram below:
$$
     \begin{tikzcd}[column sep=scriptsize, row sep=scriptsize]
  T^{\ast}Q\times \mathbb{R}
  \arrow[rr, "V_{h \cdot\delta X_t}"]
  \arrow[dd, "\pi_{\widehat{Q}}"] & & T(T^{\ast}Q\times \mathbb{R})
  \arrow[dd, "T\pi_{\widehat{Q}}"] \\
  && &\\
   \widehat{Q}\triangleq Q\times \mathbb{R}  
   \arrow[rr, "V_{h \cdot\delta X_t}^{\widehat{\Gamma}}"]\arrow[uu, bend left=50, "\widehat{\Gamma}_t"]
   & 
      & T(Q\times \mathbb{R}).
\end{tikzcd}
$$

\begin{theorem} (Stochastic contact HJ formalism II) \label{SCHJT-2}
Let $M=T^\ast Q\times \mathbb{R}$ with the canonical contact structure $\eta$. Consider the stochastic Hamiltonian system given in \eqref{SJHE}. Set $d_q f(Q\times \mathbb{R})$ be a coisotropic submanifold of $(M,\eta)$ and $(\rho \circ d_q f)(Q)$ be a Lagrangian submanifold of $(T^{\ast}Q,\omega_Q)$, for some $f\in C^\infty(Q\times \mathbb{R})$.  Let $\widehat{S}:(Q\times \mathbb{R})\times [0,{\tau}_{\widehat{Q}}) \to \mathbb{R}$ with $\widehat{S}_0(q,u)=f(q,u)$ and ${\tau}_{\widehat{Q}}$ the maximal stopping time of \eqref{Q-curve2}. Then, for each $t\in [0,{\tau}_{\widehat{Q}})$, the following statements are equivalent:
\begin{itemize}
    \item[($i$)] for every semimartingale $\hat{q}(t)$ in $\widehat{Q}=Q\times \mathbb{R}$ satisfying
    \begin{align}\label{Q-curve2}
    d \hat{q}(t) = \big( \pi_{\widehat{Q}\ast} \big) \mathcal{H}\big(X_t ,d_q\widehat{S}(\hat{q}(t),t)\big) \delta  X_t,
    \end{align}
    the process $(\omega,t)\mapsto d_q\widehat{S}(\hat{q}(t),t)$ is a Hamiltonian semimartingale solving \eqref{SJHE}. 
    \item[($ii$)] $\widehat{S}$ satisfies the stochastic contact HJ equation
     \begin{align}\label{SHJE2-1}
    \left[\pt_t d_q\widehat{S} dt- h\big(d_q\widehat{S}\big)\mathcal{R}\big(d_q\widehat{S}\big) \cdot \delta  X_t\right]^\mathbf{v} + \left[d \left( h\big(d_q\widehat{S}\big) \right) \cdot \delta {X_t} \right]^{\hat{\mathbf{v}}}=0,
    \end{align}
    $$
    \left[\pt_t (\rho \circ (d_q\widehat{S})) dt\right]^{\mathbf v} +\left[d(h\circ (d_q\widehat{S}))\right]^{\hat{\mathbf v}} \cdot \delta X_t  = \left[ (h\circ \widehat{\Gamma})\left(\iota_{\mathcal{R}}(d( (d_q\widehat{S})^\ast\theta_Q))\right)\right]^{\mathbf v} \cdot \delta X_t,
    $$
    where the operator $[\cdot]^{\hat{\mathbf{v}}}$ is defined by \eqref{VL}. In Darboux coordinates, equation \eqref{SHJE2-1} can be rewritten as
     \begin{align}\label{SHJE2-2}
     \frac{\partial \widehat{S}}{\partial q^i}
     =\frac{\partial \widehat{S}_0}{\partial q^i}-\int_0^t \Big[
     \frac{\partial h}{\partial q^i}+\frac{\partial h}{\partial p_j}\frac{\partial^2 \widehat{S}}{\partial q^i\partial q^j}+\Big( \frac{\partial h}{\partial p_j}\frac{\partial^2 \widehat{S}}{\partial q^i \partial u}+\frac{\partial h}{\partial u}\Big)\frac{\partial \widehat{S}}{\partial q^i}-h\frac{\partial^2 \widehat{S}}{\partial q^i \partial u}
     \Big]\delta  X_t
     \end{align}
with $h=h\big(q,\frac{\partial \widehat{S}}{\partial q}(q,u,t),u\big)$.
\end{itemize} 
\end{theorem}

\begin{proof}
Note that the formalisms in this case are more complex than those in Theorem \ref{SCJH-1}, for convenience and readability, we are only going to discuss the problem in Darboux coordinates, and then write down the results in a more global way. 

In Darboux coordinates, we define the section $\widehat{\Gamma}:(Q\times \mathbb{R})\times \mathbb{R}_+\to (T^\ast Q\times \mathbb{R})\times \mathbb{R}_+$ by $\widehat{\Gamma}=d_q \widehat{S}$, that is,
\begin{align}\label{local-Gamma-2}
  \widehat{\Gamma}(q^i,u,t)=\big(q^i,\widehat{\Gamma}_j(q,u,t),u,t\big)=\left(q^i,\frac{\partial \widehat{S}}{\partial q^j}(q,u,t),u,t\right).
\end{align}
Note that the right-hand side of \eqref{Q-curve2} is indeed $V_{h \cdot\delta X_t}^{\widehat{\Gamma}}=T\pi_{\widehat{Q}}\circ V_{h \cdot\delta X_t}\circ \widehat{\Gamma}_t$ for each fixed $t$, with $V_{h \cdot\delta X_t}$ being the random Hamiltonian vector field of \eqref{SJHE3} and locally given in \eqref{local-X-H}. By introducing the time-dependent vector fields $\widetilde{V}_{h \cdot\delta X_t}=V_{h \cdot\delta X_t}+\frac{\partial}{\partial t}$ and $\widetilde{V}_{h \cdot\delta X_t}^{\widehat{\Gamma}}=V_{h \cdot\delta X_t}^{\widehat{\Gamma}}+\frac{\partial}{\partial t}$, the conclusion (1) means that 
\begin{align}\label{Gamma-related-2}
\widetilde{V}_{h \cdot\delta X_t}\circ \widehat{\Gamma}=T\widehat{\Gamma} \circ \widetilde{V}_{h \cdot\delta X_t}^{\widehat{\Gamma}}.
\end{align}
  We calculate that
\begin{align}\label{TG}
&T\widehat{\Gamma}(\widetilde{X}_{h \cdot\delta X_t}^{\widehat{\Gamma}})
=\frac{\partial h}{\partial p_i} \delta  X_t T\widehat{\Gamma}\left(\frac{\partial}{\partial q^i}\right)
+\left(p_i\frac{\partial h}{\partial p_i} \delta X_t -h \cdot\delta X_t\right)T\widehat{\Gamma}\left(\frac{\partial}{\partial u}\right)+T\widehat{\Gamma}\left(\frac{\partial}{\partial t}\right)
\notag\\
=&\left( \frac{\partial h}{\partial p_i}\delta X_t \right)\frac{\partial}{\partial q^i}
+\Bigg\{\frac{\partial \widehat{\Gamma}_j}{\partial t}+\left[\frac{\partial h}{\partial p_i}\frac{\partial \widehat{\Gamma}_j}{\partial q^i} +\left(\widehat{\Gamma}_i\frac{\partial h}{\partial p_i}-h\right)\frac{\partial \widehat{\Gamma}_j}{\partial u} \right]\delta X_t \Bigg\}\frac{\partial}{\partial p_j}
\notag\\
&+\left[\left(\widehat{\Gamma}_i\frac{\partial h}{\partial p_i}-h\right)\delta X_t  \right]\frac{\partial}{\partial u}+\frac{\partial}{\partial t}.
\end{align}
Hence, the $\widehat{\Gamma}$-related formula \eqref{Gamma-related-2} holds, if and only if,
$$
-\left[\left(\frac{\partial h}{\partial q^j}+\widehat{\Gamma}_j\frac{\partial h}{\partial u}\right)\cdot\delta X_t\right]\frac{\partial}{\partial p_j}
    = \Bigg\{\frac{\partial \widehat{\Gamma}_j}{\partial t}+\left[\frac{\partial h}{\partial p_i}\frac{\partial \widehat{\Gamma}_j}{\partial q^i} +\left(\widehat{\Gamma}_i\frac{\partial h}{\partial p_i}-h\right)\frac{\partial \widehat{\Gamma}_j}{\partial u} \right]\cdot\delta X_t \Bigg\}\frac{\partial}{\partial p_j},
$$
that is,
\begin{align}\label{SHJ2-form}
\Bigg\{\frac{\partial \widehat{\Gamma}_j}{\partial t}+\left[\frac{\partial h}{\partial q^j}+\frac{\partial h}{\partial p_i}   \frac{\partial \widehat{\Gamma}_j}{\partial q^i}+\widehat{\Gamma}_j\frac{\partial h}{\partial u}+\left( \widehat{\Gamma}_i\frac{\partial h_0}{\partial p_i}-h\right)\frac{\partial \widehat{\Gamma}_j}{\partial u} \right]\cdot\delta X_t \Bigg\}\frac{\partial}{\partial p_j}=0.
\end{align}
On the other hand, under our assumptions, $\widehat{\Gamma}_t(Q\times \mathbb{R})$ is a coisotropic submanifold and it is foliated by Lagrangian leaves $\rho \circ \widehat{\Gamma}_t(Q,u)$ for each $t\in[0,\hat{\tau})$ and $u\in \mathbb{R}$. Referring to \cite[Proposition 4]{deLeon2019JMP} (see also \cite{deLeon2021contact}), the submanifold $\widehat{\Gamma}_t(Q\times \mathbb{R})$ can be given locally by the zero set of functions $\phi_a:U \to \mathbb{R}$, with $a=1,\cdots, 2n+1-\mathrm{dim}(\widehat{\Gamma}_t(Q\times \mathbb{R}))$, and it is coisotropic if and only if $\Lambda^\sharp(d\phi_a)(\phi_b)=0$ for all $a,b$. In our case, $\widehat{\Gamma}(Q\times \mathbb{R})$ is locally defined by the functions 
\begin{align}
\phi_j=p_j-\widehat{\Gamma}_j=0, \quad j=1,\cdots,n.
\end{align}
The coisotropic condition is thus equivalent to
\begin{align}\label{TG2}
\frac{\partial \widehat{\Gamma}_i}{\partial q^j}-\widehat{\Gamma}_j \frac{\partial \widehat{\Gamma}_i}{\partial u}- \frac{\partial \widehat{\Gamma}_j}{\partial q^i}+ \widehat{\Gamma}_i \frac{\partial \widehat{\Gamma}_j}{\partial u}=0.
\end{align}
Furthermore, the condition that $\rho \circ \widehat{\Gamma}(Q,u)$ is a Lagrangian  submanifold for any fixed $u\in\mathbb{R}$ yields:
\begin{align}\label{TG3}
\frac{\partial \widehat{\Gamma}_i}{\partial q^j}- \frac{\partial \widehat{\Gamma}_j}{\partial q^i}=0.
\end{align}
As a result, applying \eqref{TG2} and \eqref{TG3}, equation \eqref{SHJ2-form} becomes
\begin{align}\label{SHJ2-local}
\Bigg\{ \frac{\partial \widehat{\Gamma}_j}{\partial t}+\left[\frac{\partial h}{\partial q^j}+\frac{\partial h}{\partial p_i}   \frac{\partial \widehat{\Gamma}_i}{\partial q^j}+\widehat{\Gamma}_j\left(  \frac{\partial h}{\partial p_i}  \frac{\partial \widehat{\Gamma}_i}{\partial u} +\frac{\partial h}{\partial u}\right)-h \frac{\partial \widehat{\Gamma}_j}{\partial u} \right]\cdot\delta X_t \Bigg\}\frac{\partial}{\partial p_j}
 =0,
\end{align}
which is indistinctly referred to as an HJ equation with respect to the section $\widehat{\Gamma}$. 

In a more global language, we define
\begin{align}\label{VL}
  \alpha^{\hat{\mathbf v}} = \left( (\pi_Q)_* \alpha \right)^{\mathbf v} -\alpha(\mathcal{R}) \Lambda_Q,
\end{align}
for $\alpha\in \Omega^1(M)$, where $\alpha^{\mathbf v}$ stands for the traditional vertical lift, $\Lambda_Q =(\widehat{\Gamma}^\ast\theta_Q)^{\mathbf v}$ is the Liouville vector field, and $\mathcal{R}=\frac{\partial}{\partial u}$ is the Reeb vector field. Note that
\begin{align}
    d(h\circ \widehat{\Gamma})=\left(\frac{\partial h}{\partial q^j}+\frac{\partial h}{\partial p_i}   \frac{\partial \widehat{\Gamma}_i}{\partial q^j}\right)dq^j+\left(  \frac{\partial h}{\partial p_i}  \frac{\partial \widehat{\Gamma}_i}{\partial u} +\frac{\partial h}{\partial u}\right)du.
\end{align}
Therefore, keeping the coisotropic condition in mind, we can rewrite \eqref{SHJ2-local} into
\begin{align}
  \left[\pt_t (\rho \circ \widehat{\Gamma}) dt - (h\circ \widehat{\Gamma})\left(\iota_{\mathcal{R}}(d( \widehat{\Gamma}^\ast\theta_Q))\right) \cdot \delta X_t\right]^{\mathbf v} +\left[d(h\circ \widehat{\Gamma}) \cdot \delta X_t\right]^{\hat{\mathbf v}} = 0,
\end{align}

where $\rho \circ \widehat{\Gamma}=\widehat{\Gamma}_jdq^j$ satisfies the following commutative diagram:
    $$
    \centering
    \begin{tikzcd}[column sep=scriptsize, row sep=scriptsize]
  (T^\ast Q\times \mathbb{R}) \times \mathbb{R}_+
  \arrow[ddrr, "{id_{\mathbb{R}}\times\pi_+}"]
  \arrow[dd, "{\pi_{\widehat{Q}}\times id_+}"] & & \\
  && &\\
   (Q \times \mathbb{R})\times \mathbb{R}_+ 
   \arrow[uu, bend left=35, "\widehat{\Gamma}"]
   \arrow[rr, swap, ""]
   & & 
   \mathbb{R}\times \mathbb{R}_+,
   \arrow[uull, swap, bend left=-40, "(\rho \circ \widehat{\Gamma})"]
\end{tikzcd}
$$
and $\mathcal{R}(\rho \circ \widehat{\Gamma})=\iota_\mathcal{R}(d(\rho \circ \widehat{\Gamma}))$. The results follow by applying $\widehat{\Gamma}=d_q \widehat{S}$. 
\end{proof}

Note that the characterization of conditions on the submanifolds $\widehat{\Gamma}_t(Q\times \mathbb{R})$ and $\rho \circ \widehat{\Gamma}_t(Q,u)$ are given in \eqref{TG2} and \eqref{TG3}. The assumptions of Theorem \ref{SCHJT-2} on coisotropic submanifold and Lagrangian submanifold can be replaced by \eqref{Conditions-coisotropic} based on the following lemma.
\begin{lemma}\label{coisotropic-cond} (\cite[Theorem 2]{deLeon2021contact}) Let $\widehat{\Gamma}$ be a section of $T^\ast Q\times \mathbb{R}$ over $Q\times \mathbb{R}$. Then, $\widehat{\Gamma}(Q\times \mathbb{R})$ is a coisotropic submanifold, and $\rho \circ \widehat{\Gamma}(Q,u)$ is a Lagrangian submanifold for all $u$, if and only if,  
\begin{equation}
    d_q\widehat{\Gamma}=0 
    \text{ and } \mathscr{L}_{\mathcal{R}}\widehat{\Gamma}=\sigma\widehat{\Gamma},
\end{equation}
for some function $\sigma:Q\times \mathbb{R}\to \mathbb{R}$. That is, there exists locally a function $f:Q\times \mathbb{R}\to \mathbb{R}$ such that
\begin{equation}\label{Conditions-coisotropic}
    d_qf=\widehat{\Gamma}
    \text{ and } 
    d_q\mathcal{R}(f)=\sigma d_qf.
\end{equation}
\end{lemma}

\begin{remark}
    We next offer a view to understand these two kinds of HJB equations with respect to a contact structure. Given a time-dependent Hamiltonian $H:J^1Q\times \mathbb{R}_+\to\mathbb{R}$, which is indeed the random Hamiltonian $h \cdot\delta X_t$ in our case. The time-dependent HJ equation for the generating function $S:Q\times \mathbb{R}_+\to\mathbb{R}$ is given by
    \begin{equation}\label{time-dependent HJ}
        \frac{\partial S}{\partial t}+H \left( q,\frac{\partial S}{\partial q},S,t \right)=0.
    \end{equation}
    Let $f$ be a smooth function on $Q\times \mathbb{R}\times \mathbb{R}_+$ such that $f(q,u(t),t)=S(q,t)$. Let $\gamma=d_qf$ be a section of $J^1Q\times \mathbb{R}_+$ over $Q\times \mathbb{R}\times \mathbb{R}_+$. Then $$\gamma_i(q,u(t),t)=\frac{\partial f}{\partial q^i}(q,u(t),t)=\frac{\partial S}{\partial q^i}(q,t)=p_i(t).
    $$
    We differentiate both sides of \eqref{time-dependent HJ} with respect to $q^j$ and apply the contract Hamiltonian equation. Then we get
    \begin{align}\label{time-dependent HJ-section}
        0=&\frac{\partial}{\partial t}\frac{\partial S}{\partial q^j}+\frac{\partial H}{\partial q^j}+\frac{\partial H}{\partial p_i}\frac{\partial^2 S}{\partial q^i\partial q^j}+\frac{\partial H}{\partial u}\frac{\partial S}{\partial q^j}\notag\\
        =&\frac{\partial}{\partial t}[\gamma_j(q,u(t),t)]+\frac{\partial H}{\partial q^j}+\frac{\partial H}{\partial p_i}\frac{\partial \gamma_j}{\partial q^i}+\frac{\partial H}{\partial u}\gamma_j\notag\\
        =&\frac{\partial \gamma_j}{\partial t}+\frac{\partial \gamma_j}{\partial u}\dot{u}+\frac{\partial H}{\partial q^j}+\frac{\partial H}{\partial p_i}\frac{\partial \gamma_j}{\partial q^i}+\frac{\partial H}{\partial u}\gamma_j\notag\\
        =&\frac{\partial \gamma_j}{\partial t}+\frac{\partial \gamma_j}{\partial u}\left(p_i\frac{\partial H}{\partial p_i}-H\right)+\frac{\partial H}{\partial q^j}+\frac{\partial H}{\partial p_i}\frac{\partial \gamma_j}{\partial q^i}+\frac{\partial H}{\partial u}\gamma_j\notag\\
        =&\frac{\partial \gamma_j}{\partial t}+\frac{\partial \gamma_j}{\partial u}\left(\gamma_i\frac{\partial H}{\partial p_i}-H\right)+\frac{\partial H}{\partial q^j}+\frac{\partial H}{\partial p_i}\frac{\partial \gamma_j}{\partial q^i}+\frac{\partial H}{\partial u}\gamma_j,
    \end{align}
    which can be further rewritten as a time-dependent HJ equation for the section $\gamma$ under the coisotropic and Lagrangian conditions. In our case, equation \eqref{time-dependent HJ} corresponds to \eqref{SHJE-1-local-0}, and equation \eqref{time-dependent HJ-section} corresponds to \eqref{SHJ2-form} with $f=\widehat{S}$ and  $\gamma=\widehat{\Gamma}$. In other words, the second HJ equation can be formally obtained by taking the derivative with respect to the configuration variables of the first HJ equation and regarding the $(2n+1)^{st}$ variable as a function depending on time.

\end{remark}

\subsection{Stochastic Hamilton--Jacobi equations for the L.C.S. case}

we now turn to considering our problems on the framework of L.C.S. structures on the cotangent bundles \cite{Esen2021,Haller1999}. Let $(T^{\ast}Q,\omega_Q\triangleq-d\theta_Q)$ be the canonical symplectic manifold, where $Q$ is a base manifold and $\theta_Q$ is the canonical Liouville 1-form. Given a closed 1-form $\vartheta$ on $Q$, and it determines a closed semi-basic 1-form $\theta=\pi^\ast \vartheta$ with the cotangent bundle projection $\pi: T^\ast Q \to Q$. We define a 2-form 
\begin{align}\label{L-deR}
    \omega_\theta\triangleq -d_\theta (\theta_Q)\triangleq -d \theta_Q +\theta \wedge \theta_Q= \omega_Q +\theta \wedge \theta_Q
\end{align}
on the cotangent bundle $T^\ast Q$, where $d_\theta: \Omega^1(M)\to \Omega^2(M)$ is referred to as the Lichnerowicz--de Rham differential \cite{Guedira1984}. In this way, $\omega_\theta$ is a almost symplectic 2-form satisfying $d\omega_\theta=\theta\wedge \omega_\theta$, and $(T^\ast Q,\omega_\theta,\theta)$ thus determines an L.C.S. manifold with the Lee 1-form $\theta$. In short, we denote this L.C.S. manifold $(T^\ast Q,\omega_\theta,\theta)$ by simply $T_\theta^\ast Q$.

\begin{remark}
    It is important to note that all L.C.S. manifolds locally look like $T_\theta^\ast Q$ for some $Q$ and for a closed one-form $\vartheta$. And $T_\theta^\ast Q$ is conformally equivalent to a symplectic manifold if and only if $\vartheta$ is exact. 
\end{remark}

For fixed $t\in\mathbb{R}^+$, let $\bar{\Gamma}_t: Q\to T_\theta^\ast Q$ be a section of the cotangent bundle. A direct computation shows:
\begin{align}\label{vartheta-exact}
    \bar{\Gamma}_t^\ast \omega_\theta =-d_\theta \bar{\Gamma}_t,
\end{align}
that is, the pull-back of the L.C.S. structure is $d_\theta$ exact. Recall that the definition of Lagrangian submanifolds for almost symplectic manifolds is exactly the same as in the symplectic case since they are obtained at the linear level \cite{Esen2021}. According to \eqref{vartheta-exact}, the image space of $\Gamma_t$ is a Lagrangian submanifold
of $T_\theta^\ast Q$ if and only if $d_\theta \bar{\Gamma}_t=0$. As $d_\theta^2=0$, the image space of the a-form $d_\theta f$ is a Lagrangian submanifold of $T_\theta^\ast Q$ for some function $f:Q\to\mathbb{R}$.

Consider the randomized Hamiltonian vector field $V_{h \cdot\delta X_t}$ on $T_\theta^{\ast}Q$, which satisfies 
\begin{equation}
    \iota_{V_{h \cdot\delta X_t}}\omega_\theta=d_\theta (h \cdot\delta X_t) = d(h \cdot\delta X_t) - (h \cdot\delta X_t)\theta
\end{equation}
here. Assume that the following diagram is commutative:
$$
     \begin{tikzcd}[column sep=scriptsize, row sep=scriptsize]
  T_\theta^{\ast}Q
  \arrow[rr, "V_{h \cdot\delta X_t}"]
  \arrow[dd, "\pi"] & & TT_\theta^{\ast}Q
  \arrow[dd, "T\pi"] \\
  && &\\
    Q
   \arrow[rr, "V_{h \cdot\delta X_t}^{\bar{\Gamma}}"]\arrow[uu, bend left=50, "\bar{\Gamma}_t"]
   & 
      & TQ,
\end{tikzcd}
$$
and define a randomized vector
field on $Q$ as 
\begin{equation}\label{randomized vector field-LCS}
    V_{h \cdot\delta X_t}^{\bar{\Gamma}} =T\pi \circ V_{h \cdot\delta X_t} \circ \bar{\Gamma}_t.
\end{equation}
Now we are ready to write the stochastic HJ theorem for L.C.S. cotangent bundles.

\begin{theorem} (Stochastic HJ theory in the sense of L.C.S.) \label{SCJH-lcs}
Let $M=T_\theta^\ast Q$ be a canonical L.C.S. manifold with the Lee 1-form $\theta$. Consider the stochastic Hamiltonian system given in \eqref{SJHE}. Let $\bar{S}:Q\times [0,\bar{\tau}) \to \mathbb{R}$ with $\bar{S}_0=f\in C^{\infty}(Q)$ and $\bar{\tau}$ the maximal stopping time of \eqref{Q-curve-LCS}.  Then, for each $t\in [0,\bar{\tau})$, the following statements are equivalent:
\begin{itemize}
     \item[($i$)] for every semimartingale $\bar{q}(t)$ in $Q$ satisfying
    \begin{align}\label{Q-curve-LCS}
    d \bar{q}(t)=\pi_\ast  \mathcal{H}\big(X_t,d\bar{S}(\bar{q}(t),t)\big) \delta X_t,
    \end{align}
    the process $(\omega,t)\mapsto d\bar{S}(\bar{q}(t),t)$ is a Hamiltonian semimartingale solving \eqref{SJHE}. 
    \item[($ii$)] $\bar{S}$ satisfies the stochastic HJ equation
    \begin{align}\label{SHJE-LCS}
    \big[\pt_t(d\bar{S}) dt + d_\theta \left( h\left(d\bar{S}\right) \right) \cdot\delta  X_t\big]^\mathbf v = 0,
    \end{align}
which, in Darboux coordinates, can be rewritten as
     \begin{align}\label{SHJE-LCS-local}
    \frac{\partial \bar{S}}{\partial q^i}=\frac{\partial \bar{S}_0}{\partial q^i}+\int_0^t \left[
    \frac{\partial h}{\partial q^i}+\frac{\partial h}{\partial p_j}\frac{\partial^2 \bar{S}}{\partial q^i\partial q^j}-\theta_i h
    \right] \delta X_s,
    \end{align}
    with $h=h\big(q,\frac{\partial \bar{S}}{\partial q}(q,s)\big)$.
\end{itemize} 
\end{theorem}

\begin{proof}
    Let $\bar{\Gamma}_t=d\bar{S}(\bar{q}(t),t)$ where $\bar{q}(t)$ satisfies \eqref{Q-curve-LCS}. By simple calculations that are similar to \eqref{d-t-gamma}, we obtain that
    \begin{align}\label{d-gamma-bar}
        \int \alpha \left( \delta \bar{\Gamma}_t \right)
        = \int \mathcal{H}^\ast\big(X_t,\bar{\Gamma}_t\big) (d\bar{S}\circ \pi)^\ast \alpha \left(\delta X_t\right) + \alpha \left( d(\pt_t\bar{S})^{\mathbf v} \right)(\bar{q}(t),t) dt,
    \end{align}
    for any $ \alpha\in \Omega(T_\theta^\ast Q)$.  
    By the definition of Lichnerowicz--de Rham differential, in this case, we have
    \begin{align}
        \mathcal{H}^\ast (x,z) (\alpha) = \omega_\theta\big((dh)^\sharp(z),\alpha^\sharp(z)\big)+h(z)\mathcal{Z}(\alpha(z))
        =\omega_\theta\big((d_\theta h)^\sharp(z),\alpha^\sharp(z)\big)
    \end{align}
    with $\omega_\theta$ the almost symplectic 2-form, $\sharp:\Omega(M)\to \mathfrak{X}(M)$ the raising action induced by $\omega_\theta$ (see Appendix C) and $\mathcal{Z}=\theta^\sharp$ the Lee vector field. 

    Assuming that the conclusion (2) holds, we aim to prove the first one. By comparing \eqref{d-gamma-bar} with the stochastic Hamiltonian system for $\bar{\Gamma}_t$, i.e.,\eqref{SJHE-integral}, and defining a vector field $\Delta=\alpha^\sharp-((d\bar{S}\circ \pi)^\ast\alpha)^\sharp$, what we need to do is to show that
     \begin{equation}\label{Aim-lcs}
       \omega_\theta\big((d_\theta h)^\sharp (\bar{\Gamma}_t), \Delta \big) \delta  X_t = \alpha \left( d(\pt_t\bar{S})^{\mathbf v} \right) (\bar{q}(t),t) dt.
    \end{equation}
    Note that, under our assumptions, the image of $d\bar{S}_t(Q)$ is a Lagrangian submanifold of the L.C.S. manifold $T_\theta Q$ . That is, $(d\bar{S}_t)^\ast \omega_\theta=-d_\theta (d\bar{S}_t)=0$. We thus have
\begin{align}
   \omega_\theta\big((d_\theta h)^\sharp (\bar{\Gamma}_t), \Delta \big)
    &= \omega_\theta\big((d_\theta h)^\sharp (\bar{\Gamma}_t), \alpha^\sharp \big) -\omega_\theta\big((d_\theta h)^\sharp (\bar{\Gamma}_t), ((d\bar{S}\circ {\pi})^\ast\alpha)^\sharp \big)  \notag\\
    &= \omega_\theta\big((d_\theta h)^\sharp (\bar{\Gamma}_t) -\omega_\theta\big((d\bar{S}\circ \pi)_\ast(d_\theta h)^\sharp (\bar{\Gamma}_t), \alpha^\sharp \big) \notag\\
    &= \omega_\theta\big((d_\theta h)^\sharp (\bar{\Gamma}_t), (d\bar{S}\circ \pi)_\ast\alpha^\sharp (\bar{\Gamma}_t) \big) \notag\\
    &= -d_\theta h (\bar{\Gamma}_t) \cdot (d\bar{S})_\ast ({\pi}_\ast\alpha^\sharp) (\bar{\Gamma}_t)
    =-d_\theta(h\circ d\bar{S})\cdot {\pi}_\ast\alpha^\sharp (\bar{\Gamma}_t),
\end{align}
    where we have employed Lemma 5.2.5 in \cite{Marsden1978} (with respect to the almost symplectic 2-form $\omega_\theta$) in the third line, and the fact $(d\bar{S})^\ast (d_\theta h)=(d\bar{S})^\ast(dh-h\theta)=(d\bar{S})^\ast dh-(d\bar{S})^\ast h(d\bar{S})^\ast\theta=d((d\bar{S})^\ast h)-(d\bar{S})^\ast h\theta=d(h\circ d\bar{S})-(h\circ d\bar{S})\theta=d_\theta(h\circ d\bar{S})$ in the last line. Furthermore, by applying the stochastic HJ equation \eqref{SHJE-LCS}, we conclude that
    \begin{align}
        \omega_\theta\big((d_\theta h)^\sharp (\bar{\Gamma}_t), \Delta \big) \delta  X_t &= d(\pt_t\bar{S})\left({\pi}_\ast\alpha^\sharp \right) (\bar{\Gamma}_t)
        =\pi^\ast (d(\pt_t \bar{S}))\left(\alpha^\sharp \right) (\bar{\Gamma}_t) \notag\\
        &= \alpha \left( [{\pi}^\ast(d(\pt_t \bar{S}))]^\sharp \right) (\bar{\Gamma}_t) = \alpha \left( d(\pt_t\bar{S})^{\mathbf v} \right) (\bar{q}(t),t) dt, \notag
    \end{align}
    that is, equation \eqref{Aim-lcs} follows. The proof is done, as the conclusion (1) implies (2) by these arguments in the same way.
\end{proof}

\begin{remark}
    Note that, in Darboux coordinates, we have
    \begin{equation}
        \omega_\theta=dq^i\wedge dp_i+(\theta_idq^i)\wedge (p_jdq^j),\notag
    \end{equation}
    \begin{equation}
        V_{h \cdot\delta X_t}=\frac{\partial h}{\partial p_i} \cdot\delta X_t \frac{\partial}{\partial q^i}+\left[-\frac{\partial h}{\partial q^i} \cdot\delta X_t +
        \frac{\partial h}{\partial p_j} \cdot\delta X_t (\theta_jp_i-\theta_ip_j)+\theta_ih \cdot\delta X_t
        \right]\frac{\partial}{\partial p_i},\notag
    \end{equation}
    and
    \begin{equation}
        V_{h \cdot\delta X_t}^{\bar{\Gamma}}=\frac{\partial h}{\partial p_i} \cdot\delta X_t \frac{\partial}{\partial q^i}.\notag
    \end{equation}
    Take a section $\bar{\Gamma}:Q\times \mathbb{R}^+\to T_\theta^\ast Q \times \mathbb{R}^+$, $
\bar{\Gamma}(q^i,t)=\big(q^i,\bar{\Gamma}_j(q^i,t),t\big)$ 
    with 
    \begin{equation}\label{d-theta-Gamma}
        d_\theta \bar{\Gamma}=0,\quad \text{i.e.,}\quad   \frac{\partial \bar{\Gamma}_i}{\partial q^j}=\theta_j\bar{\Gamma}_i.
    \end{equation}
Then Theorem \ref{SCJH-lcs} means that the following two conditions are equivalent:
\begin{itemize}
    \item[($i$)] The time-dependent vector fields $\widetilde{X}_{h \cdot\delta X_t}=V_{h \cdot\delta X_t}+\frac{\partial}{\partial_t}$ and $\widetilde{X}_{h \cdot\delta X_t}^{\bar{\Gamma}}=V_{h \cdot\delta X_t}^{\bar{\Gamma}}+\frac{\partial}{\partial_t}$ are $\bar{\Gamma}$-related in the sense of
        \begin{align}\label{Gamma-related-t-lcs}
        \widetilde{X}_{h \cdot\delta X_t} \circ \bar{\Gamma} = T\bar{\Gamma} \circ \widetilde{X}_{h \cdot\delta X_t}^{\bar{\Gamma}},
        \end{align}
        which in coordinates reads
        \begin{equation}
            \left[-\frac{\partial h}{\partial q^i} \cdot\delta X_t + \frac{\partial h}{\partial p_j} \cdot\delta X_t (\theta_j\bar{\Gamma}_i-\theta_i\bar{\Gamma}_j)+\theta_ih \cdot\delta X_t \right]\frac{\partial}{\partial p_i}=\left(
        \frac{\partial \bar{\Gamma}_i}{\partial t}+\frac{\partial h}{\partial p_j} \cdot\delta X_t \frac{\partial \bar{\Gamma}_i}{\partial q^j}
        \right)\frac{\partial}{\partial p_i}.\notag
        \end{equation}
        \item[($ii$)] The L.C.S.-type HJ equation with respect to the section is fulfilled
        \begin{align}\label{SHJE-0-local-lcs}
    \left[\partial_t(\rho\circ\bar{\Gamma}) dt +d_\theta(h \cdot\delta X_t \circ \bar{\Gamma}_t)\right]^{\mathbf v}=0,
    \end{align}
    that is, in coordinates,
    \begin{align}
    \frac{\partial \bar{\Gamma}_i}{\partial t} dt +\frac{\partial h}{\partial q^i} \cdot\delta X_t +\frac{\partial h}{\partial p_j} \cdot\delta X_t \frac{\partial \bar{\Gamma}_j}{\partial q^i}-\theta_i h \cdot\delta X_t=0.\notag
    \end{align} 
\end{itemize}
    Here $\partial_t(\rho\circ\bar{\Gamma})$ is understood as the induced tangent vector at a point $p$ and is given by
        $$
    \centering
    \begin{tikzcd}[column sep=scriptsize, row sep=scriptsize]
T_\theta^\ast Q \times \mathbb{R}^+
  \arrow[ddrr, "{\pi}_+"]
  \arrow[dd, "{{\pi}\times id}"] & & \\
  && &\\
   Q \times \mathbb{R}^+ 
   \arrow[uu, bend left=35, "\bar{\Gamma}"]
   \arrow[rr, swap, ""]
   & & 
   \mathbb{R}^+.
   \arrow[uull, swap, bend left=-30, "(\rho\circ\bar{\Gamma})"]
\end{tikzcd}
$$
In this way, we can check that \eqref{Gamma-related-t-lcs} and \eqref{SHJE-0-local-lcs} are equivalent if and only if \eqref{d-theta-Gamma} is satisfied.
\end{remark}

\begin{remark}
    Compared with the local expressions of L.C.S. structures in Example \ref{exam-cls-local}, the L.C.S. structures here (on the cotangent bundles) are obtained by taking the change of coordinates: $$q_i|_{U_\alpha}=\tilde{q}^i_\alpha\quad \text{and}\quad p_i|_{U_\alpha}=e^{-\sigma_\alpha}\tilde{p}_i^\alpha.$$ In fact, given an open neighborhood $U_\alpha\subset M$ and a function $\sigma_\alpha \in C^\infty(U_\alpha)$ such that $(M,\bar{\omega})$ forms an L.C.S. manifold with $\bar{\omega}=e^{\sigma_\alpha}d\tilde{q}_\alpha^i\wedge d \tilde{p}_i^\alpha$, we can always choose $\theta|_{U_\alpha}=d\sigma_\alpha$ and define another almost symplectic 2-form
    \begin{align}
        \omega_\theta|_{U_\alpha}
        =&e^{\sigma_\alpha}d{q}_\alpha^i\wedge d {p}_i^\alpha=e^{\sigma_\alpha}d\tilde{q}_\alpha^i\wedge (e^{-\sigma_\alpha}d \tilde{p}_i^\alpha -e^{-\sigma_\alpha}\tilde{p}_i^\alpha d\sigma_\alpha)\notag\\
        =&d\tilde{q}_\alpha^i\wedge d \tilde{p}_i^\alpha-\tilde{p}_i^\alpha d \tilde{q}_i^\alpha \wedge d\sigma_\alpha =(\omega_Q-\theta_Q\wedge\theta)|_{U_\alpha}.\notag
    \end{align}
    
\end{remark}

\subsection{The complete family of stochastic  principal/generating functions}

Similar to the deterministic cases, we call a solution to stochastic HJ equations the \textit{stochastic principal function} or the \textit{stochastic generating function}. Notice that the main interest in the standard Hamilton--Jacobi theory lies in finding a complete family of solutions to the problem \cite{Carinena2006a}. In this section we shall discuss the notion of stochastic complete solutions for the HJ equation in the general framework of Jacobi manifolds.


\begin{definition}
    (Stochastic complete solutions) Assume that we have an SHS given by \eqref{SJHE} on a (transitive) Jacobi manifold $(M,\Lambda,E)$ fibered over a base manifold $\mathfrak{Q}$. Consider $U\subset \mathbb{R}^d$ an open set, where $r$ is the dimension of the fiber of the bundle $\pi: M \to \mathfrak{Q}$. That is, $d=\mathrm{dim}(M)-\mathrm{dim}(\mathfrak{Q})$. A measurable map $\Phi:\mathbb{R}_+\times {\bf \Omega}\times  \mathfrak{Q} \times U\to M$ is a \textit{stochastic complete solution} if 
    \begin{itemize}
        \item[($i$)] For each $t\in \mathbb{R}_+$ and $\omega\in {\bf \Omega}$, $\Phi_t:\mathfrak{Q} \times U\to M$ is a (local) diffeomorphism.
        \item[($ii$)] For each set of parameters $\kappa=(\kappa_1,\kappa_2,\cdots,\kappa_r)\in U$, the map 
        \begin{align}
            \Phi^\kappa:& \mathbb{R}_+\times {\bf \Omega}\times \mathfrak{Q} \to M\notag\\
            & (t,\omega,\mathfrak{q}) \mapsto \Phi^\kappa(t,\omega,\mathfrak{q})=\Phi_t(\mathfrak{q},\kappa_1,\kappa_2,\cdots,\kappa_d)
        \end{align}
        is a solution of the stochastic HJ equation associated with \eqref{SJHE}.
    \end{itemize}
\end{definition}
For simplicity we will assume that $\Phi_t$ is a global diffeomorphism. Consider $f_i:M\to\mathbb{R}$ given the composition of $\Phi_t^{-1}$ with the projection over the $i$-th component of $\mathbb{R}^d$. We have the following diagram:

$$
     \begin{tikzcd}[column sep=scriptsize, row sep=scriptsize]
  \mathfrak{Q} \times \mathbb{R}^d
  \arrow[rr,shift left, "\Phi_t"]
  \arrow[dd, "\rho_\kappa"] & & M
  \arrow[dd, "f_i"] 
  \arrow[ll,shift left, "\Phi_t^{-1}"]\\
  && &\\
  \mathbb{R}^d
   \arrow[rr, "\pi_i"]
   & 
      & \mathbb{R}.
\end{tikzcd}
$$
That is, the function $f_i$ is defined by  
\begin{equation}\label{first-integral}
    f_i(z)=\pi_i\circ \rho_\kappa \circ \Phi_t^{-1},
\end{equation}
where $\rho_\kappa:\mathcal{Q} \times \mathbb{R}^k\to \mathbb{R}^k$ and $\pi_i:\mathbb{R}^k\to\mathbb{R}$ are canonical projections.

\begin{lemma}\label{JB-fij}
    The functions defined in \eqref{first-integral} satisfy that
    $$
    \{f_i,f_j\}=f_iE(f_j)-f_jE(f_i),\quad \forall i,j=1,2,\cdots,r.
    $$
\end{lemma}

\begin{proof}
    Note that $f_i(z)=\kappa_i$ for given $z\in M$, we observe that $z\in \text{Im}(\Phi_t^\kappa)=\cap_{i=1}^kf_i^{-1}(\kappa_i)$. By our hypothesis, $\text{Im}(\Phi_t^\kappa)$ is a Lagrangian--Legendrian submanifold, we have
    $$
    \Lambda^\sharp(T(\text{Im}(\Phi_t^\kappa)))^\circ\subset T(\text{Im}(\Phi_t^\kappa)).
    $$
    Since $\Lambda^\sharp (df_i)$ is in $T(\text{Im}(\Phi_t^\kappa))$ and $df_j$ is in the annihilator of $T(\text{Im}(\phi_t^\kappa))$, we infer that  $\Lambda(df_i,df_j)=\Lambda^\sharp (df_i)(f_j)=0$. The result follows by the definition of Jacobi bracket.
\end{proof}

\begin{theorem}
\begin{itemize}

    \item[($i$)] If $\Phi$ is a complete solution of the stochastic HJ equation \eqref{SHJE-1} or \eqref{SHJE2-1} on $M=T^\ast Q\times \mathbb{R}$ (with $\mathfrak{Q}=Q$ or $\mathfrak{Q}=Q\times \mathbb{R}$ respectively), then there exist no linearly independent commuting set of first-integrals in involution \eqref{first-integral}.
    \item[($ii$)] If $\Phi$ is a complete solution of the stochastic HJ equation \eqref{SHJE-LCS} on $M=T_\theta^\ast Q$ (with $\mathfrak{Q}=Q$), then the functions defined in \eqref{first-integral} commute with respect to the Jacobi bracket if and only if  $\theta\equiv 0$ (i.e., $T_\theta^\ast Q$ reduces to a symplectic manifold). Otherwise, there exists no linearly independent commuting set of first-integrals in involution.
\end{itemize}
\end{theorem}

\begin{proof}
    ($i$) For the case of \eqref{SHJE-1}, it is clear that the images of the sections are integral submanifolds of $\ker(\eta)$ as they are Legendrian. Hence the Reeb vector field $\mathcal{R}$ is transverse to them and there is at least some index $0\leq j \leq k$ such that $\mathcal{R}(f_j)\neq 0$. By Lemma \ref{JB-fij}, if all the brackets $\{f_i,f_j\}$ vanish, then we have $\mathcal{R}(f_j/f_i)= 0$ which implies that $f_i$ and $f_j$ are not linearly independent for all $i,j=1,2,\cdots,r$.

    For the case of \eqref{SHJE2-1}, we can also claim that the Reeb vector field $\mathcal{R}$ is transverse to the coisotropic submanifold $\Phi_t^\kappa(Q\times \mathbb{R})$. In fact, if $\mathcal{R}$ is tangent to this submanifold, we would have $\mathcal{R}(p_i-(\Phi_t^\kappa)_i)=-\frac{\partial(\Phi_t^\kappa)_i }{\partial z}=0$. This would lead to a conflicting conclusion that $\Phi_t^\kappa$ does not depend on $z$ and it cannot be a diffeomorphism. So the result follows similarly.

    ($ii$) Note that the image of $\Phi_t^\kappa$ is a Lagrangian submanifold, that is, $d_\theta \Phi_t^\kappa=0$. The Lee vector field $\mathcal{Z}$ (which is defined by $\omega_\theta(\cdot,\mathcal{Z})=\theta$; See Appendix C) is tangent to any submanifolds of $\text{Im}(\Phi_t^\kappa)$. We can deduce that, in this L.C.S. case, $\mathcal{Z}(f_i)=0$ and $\{f_i,f_j\}=\omega_\theta(X_{f_i},X_{f_j})=0$ for all $i,j=1,2,\cdots,r$, if and only if $\theta\equiv 0$. 
\end{proof}

\subsection{Examples}

\begin{example} Consider a stochastic dissipative system $\ddot{x}+\gamma \dot{x}+ \nabla V(x)=\sigma \dot{B}_t$ on $\mathbb{R}$, as discussed in Example \ref{example2}. Assume that $j^1f(\mathbb{R})$ is a Legendrian submanifold of $\mathbb{R}^3$ for some $f\in C^\infty(\mathbb{R})$. In this case, the additional variable in the contact setting can be understood as a stochastic principal function 
$
    S=\int \left(\frac{1}{2}\dot{x}^2-V(x)-\gamma S\right)dt+\sigma x dB_t
$ 
with $S_0=f$, and $j^1S$ satisfies the stochastic contact Hamiltonian system \eqref{ex:model2}. Based on Theorem \ref{SCJH-1}, the corresponding stochastic contact HJ equation is 
\begin{equation}\label{example2-HJ}
    \frac{\partial S}{\partial t}=-\frac{1}{2}\left(\frac{\partial S}{\partial x} \right)^2-V(x)-\gamma S + \sigma x\dot{B}_t.
\end{equation}
Here the momentum variable is $y=\dot{x}=\frac{\partial S}{\partial x}$. by combining this condition with the fact $\dot{S}=\frac{\partial S}{\partial x}\dot{x}+\frac{\partial S}{\partial t}$, the equivalence between \eqref{ex:model2} and \eqref{example2-HJ} can be verified easily. Furthermore, we may regard $S=\int L(x,\dot{x},t)$ as a stochastic action functional (up to a constant) with $L$ being a randomized Lagrangian.

Additionally, we can also consider a general additional variable $u$ and choose a coisotropic section $\widehat{\Gamma}$ with local components $(x,\widehat{\Gamma}(x,u),u,t)$. Complement this section with a closed condition $d\widehat{\Gamma}_u=0$ and set $\widehat{\Gamma}_u=1$ for convenience. In this way, we can obtain an alternative stochastic HJ equation 
\begin{equation}\label{example2-HJ2}
    \frac{\partial \widehat{\Gamma}}{\partial t}=-\frac{1}{2}\widehat{\Gamma}^2-\widehat{\Gamma} \frac{\partial \widehat{\Gamma}}{\partial x}-\gamma \widehat{\Gamma} +V(x)-V^\prime (x)+\gamma u+(\sigma-\sigma x)\dot{B}_t.
\end{equation}
According to Lemma \ref{coisotropic-cond}, there exists locally a function $\widehat{S}\in C^\infty(\mathbb{R}^2\times \mathbb{R}_+)$ such that $\widehat{\Gamma}=\frac{\partial \widehat{S}}{\partial x}$. We thus can further rewrite \eqref{example2-HJ2} into one with respect to $\widehat{S}$. 
\end{example}

\begin{example}
    Consider a stochastic Gaussian isokinetic system $\dot{q}^i=p_i$, $\dot{p}_i=f_i(q)+\sigma_k^i(q)\diamond \dot{B}_t^k-\alpha \dot{q}$ on $\mathbb{R}^{2d}$, as discussed in Example \ref{ex:model3}. For fixed $t$, we start by a one-form $d\bar{S}(q^i)=\frac{\partial \bar{S}}{\partial q^j}(q^i)dq^j$ satisfying the Lagrangian condition, i.e., $$(d\bar{S})^\ast\omega_c=-d_\theta (d\bar{S})=-\left(\frac{\partial^2 \bar{S}}{\partial q^i\partial q^j }-\frac{1}{2c}\frac{\partial V}{\partial q^i}\frac{\partial \bar{S}}{\partial q^j}
    \right)dq^i\wedge dq^j=0.$$
    It is not hard to check that the condition ($i$) in Theorem \ref{SCJH-lcs} or the $d\bar{S}$-related condition \eqref{Gamma-related-t-lcs} reads
    \begin{equation}\label{Example2-HJ1}
        -\frac{\partial V}{\partial q^i}-\alpha \frac{\partial \bar{S}}{\partial q^i}+\sigma_k^i\diamond\dot{B}_t^k=\frac{\partial^2 \bar{S}}{\partial q^i\partial q^j }\frac{\partial \bar{S}}{\partial q^j }.
    \end{equation}
    On the other hand, noting that $-\frac{\partial h_k}{\partial q^i}+\theta_ih_k=\sigma_k^i$ in our setting, we can also infer that the stochastic HJ equation is
    \begin{equation}\label{Example2-HJ2}
        \frac{\partial \bar{S}}{\partial q^i}=\frac{\partial \bar{S}_0}{\partial q^i}+\int_0^t\left(p_j\frac{\partial^2 \bar{S}}{\partial q^i \partial q^j}-\frac{1}{2}\frac{\partial V}{\partial q^i}\right)ds-\sigma_k^i\delta B_s^k.
    \end{equation}
    It is not hard to check that \eqref{Example2-HJ1} and \eqref{Example2-HJ2} are equivalent if the Lagrangian condition holds.
\end{example}



\section*{Appendix}

\subsection*{(A) Symplectic manifolds}
Let $M$ be a manifold of even dimension $m=2n$. A symplectic structure/form is a non-degenerate and closed differential 2-form. The pair $(M,\omega)$ is called a symplectic manifold.

We define the map 
\begin{equation}\label{lowering action}
\flat: \mathfrak{X}(M) \to \Omega^1(M),\; X \mapsto X^\flat=\flat(X)=\iota_X\omega=\omega(X,\cdot)
\end{equation}
 (which is indeed an isomorphism) and its inverse map 
 \begin{equation}
 \sharp:\Omega^1(M) \to \mathfrak{X}(M) ,\; \alpha \mapsto \alpha^\sharp=\flat^{-1}(\alpha).
 \end{equation}
 These two maps $\flat$ and $\sharp$ are usually referred to as the lowering and raising actions respectively \cite[Definition 3.2.1]{Marsden1978}. A symplectic manifold $(M,\omega)$ can be regarded as a Jacobi manifold $(M,\Lambda ,0)$ of necessarily even dimension whose associated Jacobi tensor given by
\begin{equation}
\Lambda(\alpha,\beta)=\omega(\alpha^\sharp,\beta^\sharp), \quad  \forall \alpha,\beta \in \Omega^1(M).
\end{equation}
The (Poisson) bracket defined by 
\begin{equation}
\{ f,g\}=\omega\left((df)^\sharp,(dg)^\sharp \right)=\langle dg, (df)^\sharp \rangle=-\langle df, (dg)^\sharp \rangle,\quad \forall f,g\in C^\infty(M),
\end{equation}
satisfies the Leibniz identity (and, of course, weaker Leibniz identity condition). 

Clearly, in this case $\Lambda^\sharp$ in \eqref{Sharp-Lambda} is just equal to $\sharp$. The Hamiltonian vector field (associated with a smooth function $h$) satisfies  
\begin{equation}
V_h=\Lambda^\sharp(dh)=(dh)^\sharp,
\end{equation}
or equivalently, $dh=(V_h)^\flat=\iota_{V_h}\omega=\omega(V_h,\cdot)$. 

\begin{example}\label{exam-local-sym} (Symplectic Hamiltonian dynamics in Darboux coordinates)

By classical Darboux theorem \cite[Theorem 3.2.2]{Marsden1978}, denoting by $(q^1,\cdots,q^n)$ the coordinates on $Q$ and by $(q^1,\cdots,q^n,p_1,\cdots,p_n)$ those on $M=T^\ast Q$, the symplectic form becomes
$$
\omega=dq^i\wedge dp_i.
$$
Let $h$ be a smooth function on $M$. The Hamiltonian vector field is 
$$
V_h=\frac{\partial h}{\partial p_i}\frac{\partial}{\partial q^i}-\frac{\partial h}{\partial q^i}\frac{\partial}{\partial p_i}.
$$

\end{example}

\subsection*{(B) Contact manifolds}
Let $M$ be a manifold of odd dimension $m=2n+1$. A contact structure/distribution is a maximally non-integrable hyperplane field $\boldsymbol{\xi}= \text{ker} \eta \subset TM$, that is, the defining differential 1-form $\eta$ is required to satisfy $\eta \wedge (d\eta)^n\neq 0$ 
(meaning that it vanishes nowhere). Such a one form $\eta$ is called a contact form. The volume form $\eta \wedge (d\eta)^n$ is called the associated contact volume element. The pair $(M,\boldsymbol{\xi})$ is commonly called a contact manifold. 

Given a contact distribution $\boldsymbol{\xi}$, its defining 1-form $\eta$ is determined up to a nowhere vanishing smooth function on  $M$. However, The condition $\eta \wedge (d\eta)^n\neq 0$ is independent of the specific choice of $\eta$. To study the dynamics of contact Hamiltonian systems, it is more convenient to consider contact manifolds with a fixed choice of contact forms \cite{Geiges2008,deLeon2019JMP}. This leads to another widespread definition of contact manifolds. That is, a contact manifold is a pair $(M,\eta)$ where $M$ is a $(2n+1)$-dimensional manifold and $\eta$ is a contact form.

Given a contact form $\eta$, we can define an isomorphism 
\begin{equation}\label{flat-contact}
 \flat: \mathfrak{X}(M) \to \Omega^1(M),\; X \mapsto X^\flat=\flat(X)=\iota_X d\eta +\eta(X)\eta
\end{equation}
and its inverse map is denoted by $\sharp=\flat^{-1}$.  
In particular, the vector field
\begin{equation}\label{Reeb vector field}
\mathcal{R}=\eta^\sharp=\flat^{-1}(\eta)
\end{equation}
is called the Reeb vector field of $M$, which is the unique vector field such that $\iota_{\mathcal{R}}d\eta=0$ and $\eta(\mathcal{R})=1$. The triple $(M,\Lambda, -\mathcal{R})$ is a Jacobi manifold of necessarily odd dimension with the associated tensor given by
\begin{equation}\label{Lambda-contact}
\Lambda (\alpha,\beta)= -d\eta(\alpha^\sharp,\beta^\sharp), \quad  \forall \alpha,\beta \in \Omega^1(M). 
\end{equation}
The (contact) bracket defined by
 \begin{equation}
  \{f,g\}=-d\eta\left((df)^\sharp,(dg)^\sharp \right)-f\mathcal{R}(g)+g\mathcal{R}(f), \quad  \forall f,g\in C^\infty(M),
 \end{equation}
satisfies the weaker Leibniz identity condition but does not satisfy Leibniz identity.

In this case, noting that $\beta=\flat(\beta^\sharp)=\iota_{\beta^\sharp} d\eta +\eta(\beta^\sharp)\eta$,
we have
\begin{align}
\Lambda^\sharp(\alpha)(\beta)&=-d\eta(\alpha^\sharp,\beta^\sharp)=\iota_{\beta^\sharp} d\eta(\alpha^\sharp)=\beta(\alpha^\sharp)-\eta(\beta^\sharp)\eta(\alpha^\sharp) \notag\\
&=\langle \beta,\alpha^\sharp \rangle-\langle \eta,\beta^\sharp \rangle \langle \eta,\alpha^\sharp \rangle =\langle \alpha^\sharp, \beta \rangle-\langle \beta,\eta^\sharp \rangle \langle \alpha,\eta^\sharp \rangle\notag\\
&=\sharp(\alpha)(\beta)-\alpha(\mathcal{R})\beta(\mathcal{R}). \notag
\end{align}
That is 
\begin{equation}\label{sharp-contact}
\Lambda^\sharp(\alpha)=\sharp(\alpha)-\alpha(\mathcal{R})\mathcal{R},
\end{equation}
which is not an isomorphism \cite[Remark 2]{deLeon2019JMP}.
Furthermore, we can obtain that 
\begin{proposition}\label{Proposition-hvf}
Let $(M,\eta)$ be a contact manifold. Given a smooth real function $h$ on $M$. The Hamiltonian vector fields satisfies the following equivalent properties:

(i) $V_h=\Lambda^\sharp(dh)-h\mathcal{R}=(dh)^\sharp-[\mathcal{R}(h)+h]\mathcal{R}$;\par
(ii) $(V_h)^\flat=dh-[\mathcal{R}(h)+h]\eta$;\par
(iii) $h=-\eta(V_h)$ and $dh=\iota_{V_h} d\eta+\mathcal{R}(h)\eta$;\par
(iv) $h=-\eta(V_h)$ and $\mathscr{L}_{V_h}\eta=f_h \eta$ with a function $f_h: M\to \mathbb{R}$ that is completely determined by $h$. In our settings, we can show that $f_h= -\mathcal{R}(h)$.
\end{proposition}

\begin{proof} (1) The property ($i$) follows from the fact $M$ is a Jacobi manifold with \eqref{sharp-contact}. That is, by \eqref{Hamiltonian-vf} and \eqref{sharp-contact}, we have
$$
V_h=\Lambda^\sharp(dh)-h\mathcal{R}=\sharp(dh)-dh(\mathcal{R})\mathcal{R}-h\mathcal{R}=\sharp(dh)-[\mathcal{R}(h)+h]\mathcal{R}.
$$
Here we have also used the fact that $dh(\mathcal{R})=h^\ast \mathcal{R} =\mathcal{R}(h)$. (2) The property ($ii$) and its equivalence with ($i$) follow immediately as $\flat=\sharp^{-1}$. (3) By the  property ($ii$) and definition \eqref{flat-contact}, we obtain that
$$
-h=\flat(V_h)(\mathcal{R})=\iota_{V_h} d\eta(\mathcal{R}) +\eta(V_h)\eta(\mathcal{R})=\eta(V_h),
$$
and
$$
dh=\flat(\sharp(dh))=\iota_{\sharp(dh)}d\eta+\eta(\sharp(dh))\eta=\iota_{V_h} d\eta+\mathcal{R}(h)\eta.
$$
It is thus clear that ($iii$) is equal to ($ii$). (4) If the  property ($iii$) holds, then by Cartan's formula, 
$$
\mathscr{L}_{V_h}\eta=d\iota_{V_h}\eta+\iota_{V_h}d\eta=d\eta(V_h)+\iota_{V_h}d\eta=-dh+\iota_{V_h}d\eta=-\mathcal{R}(h)\eta.
$$
Hence the results in ($iv$) follow with $f_h=-\mathcal{R}(h)$. Conversely, if ($iv$) holds, then we can calculate that 
$$
dh=\iota_{V_h}d\eta-\mathscr{L}_{V_h}\eta=\iota_{V_h}d\eta-f_h \eta.
$$
Note that $dh(\mathcal{R})=\iota_{V_h}d\eta(\mathcal{R})-f_h \eta(\mathcal{R})=-f_h$. The function $f_h$ is nothing other than $-\mathcal{R}(h)$. We thus conclude that ($iv$) is equal to ($iii$). The proof is completed.
\end{proof}


\begin{example}\label{exam-local-chart} (Contact Hamiltonian dynamics in Darboux coordinates) \\
By Darboux theorem \cite[Theorem 2.5.1]{Geiges2008} for contact manifold, around each point in $M=T^\ast Q\times \mathbb{R}$, we can find local coordinates $(q^i,p_i,u)$, $i=1,\cdots,n,$ such that 
\begin{equation}
\eta=du-p_i dq^i,\quad d\eta=dq^i\wedge dp_i,\;\text{ and }\;\mathcal{R}=\frac{\partial}{\partial u}.\notag
\end{equation}
 Given a smooth function $h\in C^\infty(M)$, we can associate it with the contact Hamiltonian vector field
  \begin{equation}
 V_h=\frac{\partial h}{\partial p_i}\frac{\partial}{\partial q^i}-\left(\frac{\partial h}{\partial q^i}+p_i\frac{\partial h}{\partial u}\right)\frac{\partial}{\partial p_i}+\left(p_i\frac{\partial h}{\partial p_i}-h\right)\frac{\partial}{\partial u}.\notag
 \end{equation}
According to Proposition \ref{Proposition-hvf} ($i$), we can also calculate that
\begin{align*}
\Lambda^\sharp(dh)&=\frac{\partial h}{\partial p_i}\left(\frac{\partial }{\partial q^i}+p_i\frac{\partial  }{\partial u} \right)-\left(\frac{\partial h}{\partial q^i}+p_i\frac{\partial h}{\partial u} \right)\frac{\partial }{\partial p_i}\notag\\
&=B^i(h)A_i-A_i(h)B^i
\end{align*}
and 
$$
(dh)^\sharp=B^i(h)A_i-A_i(h)B^i-\mathcal{R}(h)\mathcal{R},
$$
where
$$A_i=\left(\frac{\partial }{\partial q^i}+p_i\frac{\partial}{\partial u} \right)\quad \text{ and }\quad B^i=\frac{\partial}{\partial p_i}.
$$
Note that the above local vector fields satisfy $d\eta(A_i,A_j)=d\eta(B^i,B^j)=0$, $d\eta(A_i,B^j)=\delta_{ij}$, and $\{A_1,B^1,\cdots,A_n,B^n,\mathcal{R}\}$ forms a dual basis corresponding to $\{dq^1,dp_1,\cdots, dq^n,dp_n,\eta\}$ \cite[Theorem 2]{deLeon2019JMP}. But this basis is not a coordinate basis of any chart, since the Lie brackets $[A_i,B^i]=-\mathcal{R}$ ($\forall i=1,\cdots,n$) which do not vanish. Moreover, we obtain the following distributions
\begin{align*}
&\boldsymbol{\xi}=\text{ker}\,\eta=<A_1,B^1,\cdots,A_n,B^n>,\\
&\boldsymbol{\xi}^v=\text{ker}\,d\eta=<\mathcal{R}>,
\end{align*}
which are referred to as horizontal and vertical distributions respectively, and satisfy the Whitney sum decomposition $TM=\boldsymbol{\xi} \oplus \boldsymbol{\xi}^v.$
\end{example}


\subsection*{(C) Locally Conformally Symplecyic (L.C.S.) manifolds}

Let $M$ be an even-dimensional manifold equipped with a nondegenerate (but not necessarily closed) 2-form $\bar{\omega}$. That is, $(M,\bar{\omega})$ is a so called almost symplectic manifold. This manifold turns out to be a symplectic one if $\bar{\omega}$ is additionally closed. The L.C.S. manifolds are regarded as an intermediate step between almost symplectic manifolds and symplectic manifolds \cite{Vaisman1985}, for which the two-form $\bar{\omega}$ is closed locally up to a conformal parameter. 

There are two equivalent definitions of L.C.S. manifolds. One is local: an almost symplectic manifold $(M,\bar{\omega})$ is said to be L.C.S., if there exists an open neighborhood $U_\alpha$ around each point $z$ in $M$, and a function $\sigma_\alpha\in C^\infty(U_\alpha)$ such that $d(e^{-\sigma_\alpha} \bar{\omega}|_\alpha)=0$, so $(U_\alpha,\omega_\alpha\triangleq e^{-\sigma_\alpha} \bar{\omega}|_\alpha)$ is symplectic manifold. And the other is global: $(M,\bar{\omega})$ is L.C.S. if there exists  a closed 1-form $\theta$ such that
 $$
 d\bar{\omega}=\theta \wedge \bar{\omega}.
 $$
 This 1-form $\theta$ is called Lee 1-form. Clearly, it is locally exact in the sense of $\theta|_{U_\alpha}=d\sigma_\alpha$. 
 If $U_\alpha=M$, that is, $\theta$ is exact, then the manifold is said to a globally conformally symplectic (G.C.S.) manifold. If $\theta=0$, then it is degenerated into a symplectic manifold. 

Similar to \eqref{lowering action}, the non-degeneracy of the 2-form $\bar{\omega}$ leads us to define a isomorphism
\begin{equation}\label{lowering action-lcs}
\flat: \mathfrak{X}(M) \to \Omega^1(M),\; X \mapsto X^\flat=\flat(X)=\iota_X\bar{\omega}=\bar{\omega}(X,\cdot)
\end{equation}
and its inverse map $\sharp=\flat^{-1}$. Given an L.C.S. manifold $(M,\bar{\omega},\theta)$, the vector field
\begin{equation}\label{anti-Lee vf}
\mathcal{Z}=\theta^\sharp=\flat^{-1}(\theta),
\end{equation}
is called Lee vector field, and satisfies $\iota_{\mathcal{Z}}\bar{\omega}=\theta$ and $\theta(\mathcal{Z})=0$. Define a bivector $\Lambda$ on $M$ as
\begin{equation}\label{Lambda-LCS}
\Lambda (\alpha,\beta)= \bar{\omega}(\alpha^\sharp,\beta^\sharp), \quad  \forall \alpha,\beta \in \Omega^1(M). 
\end{equation}
It is clear that the triple $(M,\Lambda,\mathcal{Z})$ is an even dimensional Jacobi manifold, and it is not Poisson when $\theta\neq 0$. And the bracket on an L.C.S. manifold is defined as
 \begin{equation}
\{f,g\}=\bar{\omega}\left((df)^\sharp,(dg)^\sharp \right)+f\mathcal{Z}(g)-g\mathcal{Z}(f)=\bar{\omega}\left(V_f,V_g \right), \quad \forall f,g\in C^\infty(M).
 \end{equation}
 
 Accordingly, $\Lambda^\sharp=\sharp$ and the Hamiltonian vector field is given by
\begin{equation}
V_h=(dh)^\sharp+h\mathcal{Z}.
\end{equation}

 \begin{example}\label{exam-cls-local}
 (The local expressions for Hamiltonian dynamics on the L.C.S. manifold)
 Based on the local definition of L.C.S. manifold, we consider the Darboux coordinates $(q_\alpha^i,p_i^\alpha)$ on $U_\alpha$. That is, the corresponding symplectic 2-form is $\omega_\alpha=dq_\alpha^i\wedge d p_i^\alpha.$ And the almost symplectic 2-form on $M$ is 
 $$
 \bar{\omega}=e^{\sigma_\alpha}dq_\alpha^i\wedge d p_i^\alpha,
 $$
 with $\sigma_\alpha\in C^\infty(U_\alpha)$.
 Clearly, $d\bar{\omega}=\theta \wedge \bar{\omega}=e^{\sigma_\alpha}d\sigma_\alpha\wedge dq_\alpha^i\wedge d p_i^\alpha$, where the Lee 1-form $\theta$ is given by $\theta|_{U_\alpha}=d\sigma_\alpha=\frac{\partial \sigma_\alpha}{\partial q^i}dq^i+\frac{\partial \sigma_\alpha}{\partial p_i}dp_i$.

  For a (local) Hamiltonian function $h_\alpha$ on this chart, we write down the (local) Hamiltonian vector field $X_\alpha$ by $dh_\alpha=\iota_{X_\alpha}\omega_\alpha$. To arrive at a global setting, we introduce a real valued function $h$ on the whole $M$ which is defined by
  $$
h|_{U_\alpha}=e^{\sigma_\alpha}h_\alpha.
  $$
  A direct calculation shows that
  $
\iota_{X_\alpha}\bar{\omega}|_{\alpha}=dh|_{\alpha}-h|_{\alpha}d\sigma_\alpha,
  $
  whose terms all have global realizations. We thus have
  $$
\iota_{V_h}\bar{\omega}=dh-h\theta,
  $$
  where $V_h$ is the vector field obtained by gluing all $X_\alpha$s, that is, ${V_h}|_{U_\alpha}=X_\alpha$.

  In addition, for the (local) symplectic manifold $(U_\alpha,{\omega}_\alpha)$ and two local functions $f_\alpha$, $g_\alpha$, we can define the local bracket
  \begin{equation}
\{f|_\alpha,g|_\alpha\}=e^{-\sigma_\alpha}\{e^{\sigma_\alpha}f_\alpha,e^{\sigma_\alpha}g_\alpha \}, \notag
 \end{equation}
 where the bracket on the right-hand side is the canonical Poisson bracket defined by means of the local symplectic 2-form $\omega_\alpha$, and $f|_\alpha=e^{\sigma_\alpha}f_\alpha$, $g|_\alpha=e^{\sigma_\alpha}g_\alpha$ stand for the local realizations of the two global function $f$, $g$, respectively.
 \end{example}

 \subsection*{(D) A summary table}

We summarize and compare the geometric structures of contact, cosymplectic, symplectic and L.C.S. manifolds in the following table.

\begin{table}[!htb]\label{table1}\centering
\resizebox{0.9\columnwidth}{!}{
\begin{tabularx}{\textwidth}{ p{2cm} p{3cm} p{3cm} p{2.5cm} p{3.5cm} p{0cm} }
\toprule
 Structures  & \centering Contact  & \centering Cosymplectic  & \centering Symplectic  & \centering L.C.S. & \\
\midrule
Symbols         & \centering $(M,\eta)$                           & \centering $(M,\Omega,\bar{\eta})$                             & \centering $(M,\omega)$         & \centering $(M,\bar{\omega},\theta)$ &  \\
\midrule
Characteri- zations         &  \centering  $\eta\wedge (d\eta)^n \neq 0$                           &  \centering  $d\bar{\eta}=0,d\Omega=0,$  $\eta\wedge \Omega^n \neq 0$                            & \centering  $d\omega=0$,

$\omega^n\neq 0$       & \centering  $d\theta=0$, \quad $d\bar{\omega}=\theta\wedge \bar{\omega}$ &  \\
\midrule
 Morphisms        & \centering 
 $\flat(X)=\iota_X d\eta +\eta(X)\eta,$ $\sharp=\flat^{-1}$                         & \centering  $\flat(X)=\iota_X \Omega +\bar{\eta}(X)\bar{\eta},$
 
 $\sharp=\flat^{-1}$                            &\centering 
 $\flat(X)=\iota_X \omega,$  $\sharp=\flat^{-1}$         & \centering $\flat(X)=\iota_X \bar{\omega},$  $\sharp=\flat^{-1}$ &  \\
\midrule
Induced 

Jacobi 

structures        &  \centering
 $\Lambda_\eta=-d\eta(\sharp(\cdot),\sharp(\cdot)),$ 
 
$E=\mathcal{R}=\sharp(\eta),$ 

$\Lambda^\sharp(\alpha)=\sharp(\alpha)-\alpha(\mathcal{R})\mathcal{R}$
& \centering 
$\Lambda_{\Omega,\bar{\eta}}=\Omega(\sharp(\cdot),\sharp(\cdot)),$ 

$E=0,$ 

$\bar{\mathcal{R}}=\sharp(\bar{\eta}),$

$\Lambda^\sharp \neq \sharp$ 
& \centering 
$\Lambda_{\omega}=\omega(\sharp(\cdot),\sharp(\cdot)),$ 

$E=0,$ 

$\Lambda^\sharp = \sharp$       &  \centering 
$\Lambda_{\bar{\omega}}=$

$\bar{\omega}(\sharp(\cdot),\sharp(\cdot)),$ 

$E=\mathcal{Z}=\sharp(\theta),$ 

$\Lambda^\sharp = \sharp$ &  \\

\midrule

Hamiltonians & \centering $V_h=\sharp(dh)-[\mathcal{R}(h)+h]\mathcal{R}$ 
& \centering $V_h=\sharp(dh)$

& \centering $V_h=\sharp(dh)$

& \centering $V_h=\sharp(dh)+h\mathcal{Z}$ & \\
\midrule
Local 

expressions & \centering $(q_i,p_i,u)\in T^\ast Q \times \mathbb{R}$

$\eta=du-p_idq^i,$

$\Lambda_\eta=\left(\frac{\partial}{\partial q^i}+p_i\frac{\partial}{\partial u}\right)\wedge \frac{\partial}{\partial p_i},$

$E=\mathcal{R}=\frac{\partial}{\partial u},$

$V_h= \frac{\partial H}{\partial p_i}\frac{\partial}{\partial q^i}-\left(\frac{\partial H}{\partial q^i}+p_i \frac{\partial H}{\partial z} \right)\frac{\partial }{\partial p_i}-\left(p_i \frac{\partial H}{\partial p_i} - H\right)\frac{\partial }{\partial u} $

& \centering $(q_i,p_i,t)\in T^\ast Q \times \mathbb{R}$

$\bar{\eta}=dt,$ $\Omega=dq^i\wedge dp_i,$

$\Lambda_{\Omega,\bar{\eta}}=\frac{\partial}{\partial q^i}\wedge \frac{\partial}{\partial p_i},$

$E=0,$ $\bar{\mathcal{R}}=\frac{\partial}{\partial t},$

$V_h= \frac{\partial H}{\partial p_i}\frac{\partial}{\partial q^i}+ \frac{\partial H}{\partial q^i}\frac{\partial}{\partial p_i}+\frac{\partial}{\partial t}$

& \centering $(q_i,p_i)\in T^\ast Q$

$\omega=dq^i\wedge dp_i,$

$\Lambda_\omega=\frac{\partial}{\partial q^i}\wedge \frac{\partial}{\partial p_i},$

$E=0,$

$V_h= \frac{\partial H}{\partial p_i}\frac{\partial}{\partial q^i}+ \frac{\partial H}{\partial q^i}\frac{\partial}{\partial p_i}$

& \centering $(q^i_\alpha,p_i^\alpha)\in U_\alpha\subset T^\ast Q$

$\theta=d\sigma_\alpha,$

$\bar{\omega}=e^{\sigma_\alpha}dq_\alpha^i\wedge d p_i^\alpha,$

$\Lambda_{\bar{\omega}}=e^{-\sigma_\alpha}\frac{\partial}{\partial q_\alpha^i}\wedge \frac{\partial}{\partial p_i^\alpha},$

$E=\mathcal{Z}=e^{-\sigma_\alpha}\Big(\frac{\partial \sigma_\alpha}{\partial p_i^\alpha}\frac{\partial}{\partial q_\alpha^i}-\frac{\partial \sigma_\alpha}{\partial q_\alpha^i}\frac{\partial}{\partial p_i^\alpha}\Big),$

$V_h= e^{-\sigma_\alpha}\Big(\frac{\partial H}{\partial p_i}\frac{\partial}{\partial q^i}+ \frac{\partial H}{\partial q^i}\frac{\partial}{\partial p_i}\Big)$

& \\
\bottomrule
\end{tabularx}
}
\caption{Four subclasses of the Jacobi structures.}
\end{table}


{\footnotesize
\bibliographystyle{alpha}
\bibliography{Wei-refs}
}

\end{document}